\documentclass[reqno,oneside,11pt]{amsart}
\usepackage{amsmath,amssymb,latexsym,soul,cite,mathrsfs,stmaryrd}
\usepackage{color,enumitem,graphicx}
\usepackage[colorlinks=true,urlcolor=blue,
citecolor=blue,linkcolor=blue,linktocpage,pdfpagelabels,
bookmarksnumbered,bookmarksopen]{hyperref}
\usepackage[dvipsnames,svgnames]{xcolor}
\usepackage[english]{babel}
\usepackage{geometry}
\usepackage[percent]{overpic}

\usepackage{subcaption}
\usepackage{amsthm}
\usepackage{mathtools}
\DeclarePairedDelimiter{\norm}{\lVert}{\rVert}
\DeclarePairedDelimiter{\sprod}{\langle}{\rangle}
\DeclarePairedDelimiter{\abs}{\lvert}{\rvert}

\usepackage{bbm}
\newtheorem{theorem}{Theorem}
\newtheorem{lemma}[theorem]{Lemma}

\newtheorem{proposition}[theorem]{Proposition}

\newtheorem{theoremletter}{Theorem}

\newtheorem{theoremroman}{Theorem}

\theoremstyle{definition}
\theoremstyle{remark}\newtheorem{remark}[theorem]{Remark}
\theoremstyle{definition}

\newcommand{\innerthmname}{}

\theoremstyle{definition}

\makeatletter
\def\namedlabel#1#2{\begingroup
	#2
	\def\@currentlabel{#2}
	\phantomsection\label{#1}\endgroup
}
\makeatother

\numberwithin{theorem}{section} 
\numberwithin{equation}{section}

\newtheorem{asu}{}

\title[Multiplicity of solutions for Gross-Pitaevskii equations]{Multiplicity of solutions for Gross-Pitaevskii equations on Riemannian manifolds}

\author[D. Corona]{Dario Corona}
\author[S. Nardulli]{Stefano Nardulli}
\author[R. Oliver-Bonafoux]{Ramon Oliver-Bonafoux}
\author[G. Orlandi]{Giandomenico Orlandi}

\address[D. Corona]{
	School of Science and Technology,
	University of Camerino
	\newline\indent 
62032, Camerino, Italy}
\email{\href{mailto:dario.corona@unicam.it}{dario.corona@unicam.it}}

\address[S. Nardulli]{Department of Mathematics,
	Federal University of ABC
	\newline\indent 
09210-580, S\~ao Paulo, Brazil}
\email{\href{mailto:stefano.nardulli@ufabc.edu.br}{stefano.nardulli@ufabc.edu.br}}

\address[R. Oliver-Bonafoux]{Dipartimento di Informatica,
	Universit\`a di Verona,
	\newline\indent
	37134,Verona, Italy
}
\email{\href{mailto:ramon.oliverbonafoux@univr.it}{ramon.oliverbonafoux@univr.it}}

\address[G. Orlandi]{Dipartimento di Informatica,
	Universit\`a di Verona,
	\newline\indent
	37134,Verona, Italy
}
\email{\href{mailto:giandomenico.orlandi@univr.it}{giandomenico.orlandi@univr.it}}

\date{}


\begin{document}

\begin{abstract}
	We provide a multiplicity result for solutions of time-independent
	Gross-Pitaevskii equations on closed Riemannian manifolds.
	Such solutions arise as (possibly non-minimizing) critical points of the
	Ginzburg-Landau energy having prescribed momentum according to a given
	tangent velocity field. Lower bounds on the multiplicity of solutions are
	obtained in terms of the topology of the maximum velocity set, in the small
	momentum and  vorticity core size regime.
	The proof relies on methods from
	critical point theory and $\Gamma$–convergence for Ginzburg-Landau
	functionals as well as on some new results for codimension 2
	isoperimetric-type problems in the small flux regime, possibly of
	independent interest.
\end{abstract}

\maketitle

\section{Introduction and Main Result}

Given a smooth closed and oriented Riemannian manifold $(M,g)$ of dimension $N \ge 3$,
let $X$ be a smooth solenoidal tangent vector field on $M$ and consider the following (time-independent) \textit{Gross-Pitaevskii} equation
\begin{equation}
	\label{eq:Ginzburg-PDE}\tag{GP$_{\varepsilon,X}$}
	\lambda\,|\log\varepsilon|\, i\, g(X,\nabla \psi)
	= \Delta \psi - \frac{1}{\varepsilon^2}\nabla W(\psi)
	\quad \text{on } M,
\end{equation}
where $\psi \colon M  \to \mathbb{C}$ is a \textit{wavefunction},
$\varepsilon>0$ is a small parameter (a characteristic length),
$\lambda\in\mathbb{R}$ plays the role of a Lagrange multiplier and
$W\colon \mathbb{C}\to[0,+\infty)$ is a $C^2$ potential of
Ginzburg–Landau type, namely it vanishes only on the unit circle $\mathbb{S}^1$ (the typical choice is $W(z)=\tfrac{1}{4}(1-|z|^2)^2$).
Equation \eqref{eq:Ginzburg-PDE} was considered in the Euclidean setting in \cite{BOS2004} and subsequently by Chiron \cite{chiron2005} to analyse the \textit{profiles} of the \textit{traveling waves} (\textit{solitons}) for the \textit{time-dependent} Gross-Pitaevskii equation
\begin{equation}\label{eq:GP_RN}
	i \, \lvert \log \varepsilon \rvert \Psi_t + \Delta \Psi = \frac{1}{\varepsilon^2}\nabla W(\Psi), \quad \Psi: \mathbb{R}^N \times \mathbb{R} \to \mathbb{C}.
\end{equation}
Equation \eqref{eq:GP_RN} is a Schrödinger-type equation modeling a variety of phenomena such as superfluidity and nonlinear optics. 
For instance, when the space domain $\mathbb{R}^N$ is replaced by $M$ (see equation \eqref{eq:GrossPitaevskii} below) the model describes superfluidity in thin shells.
Existence of traveling waves for \eqref{eq:GP_RN} was first proven by Bethuel and Saut \cite{BethuelGravejatSaut2009} (in dimension 2) and in \cite{BOS2004} as well as by Chiron \cite{chiron2004} in dimension larger than 3, see also Mari\c{s} \cite{Maris2013}.
In dimension $1$, traveling waves for \eqref{eq:GP_RN} have been extensively studied by Bethuel, Gravejat, Saut and Smets \cite{BethuelGravejatSautSmets2008,BethuelGravejatSaut2009,BethuelGravejatSmets2015,GravejatSmets2015}. Equation \eqref{eq:Ginzburg-PDE} (or closely related variants) has also been considered in $\mathbb{R}^3$  to describe vorticity patterns for rotating \textit{Bose-Einstein condensates}, mostly in the spirit of the seminal book by Bethuel, Brezis and Hélein \cite{BBH}; in such case, the tangential velocity vector field $X$ is typically a multiple of $(-y,x,0)$, corresponding to a rotation around the $z$–axis.
Without attempting to be exhaustive,
for a broader overview of vortex structures and rotating 
condensates we refer to the book by Aftalion \cite{AftalionBook}, based on previous contributions such as \cite{AftalionAlamaBronsard2005, AftalionJerrard2003,AftalionQu,AftalionRiviere,AlamaBronsard2005,IgnatMillotJFA,IgnatMillotRMP}.
In this paper we do not restrict the study to a velocity field $X$ representing a constant translation or rotation in a Euclidean space, but rather take a more general field on a Riemannian manifold.
Note that, since $X$ is divergence-free, \eqref{eq:Ginzburg-PDE} is equivalent to
\begin{equation}\label{eq:GL_equivalent}
	\left(\nabla -i \, \lambda |\log \varepsilon | \frac{X}{2} \right)^2\psi = \frac{1}{\varepsilon^2}\nabla_z \hat{W}(x,\psi) \quad \text{ on } M,
\end{equation}
where $\hat{W}(x,z)=W(z)+\varepsilon^2 | \log \varepsilon |^2 \lambda^2 \frac{\norm{X(x)}^2}{4}z$.
Equation \eqref{eq:GL_equivalent} is related to the
system arising in the Ginzburg-Landau theory of superconductivity, with $X$
playing the role of a gauge field.

Equation \eqref{eq:Ginzburg-PDE} is variational: it can be seen as the
Euler-Lagrange equation of the Gross-Pitaevskii energy $F_{\varepsilon}:
W^{1,2}(M,\mathbb{C}) \times \mathbb{R} \mapsto \mathbb{R}$, defined as
\begin{equation*}
	F_{\varepsilon}(u,\lambda):= E_{\varepsilon}(u)-4\lambda\Phi_X(u),
\end{equation*}
where $E_{\varepsilon}: W^{1,2}(M,\mathbb{C}) \mapsto \mathbb{R}$ is given by
\begin{equation}
	\label{eq:Ginzburg-Landau}
	E_{\varepsilon}(u) \coloneqq
	\frac{1}{\pi|\log\varepsilon|}
	\int_M\left( \frac{1}{2}\norm{\nabla u}^2
		+ \frac{1}{\varepsilon^2}W(u)
	\right)\mathrm{d}v_g,
\end{equation}
which is the (rescaled) Ginzburg-Landau energy, while $\Phi_X\colon
W^{1,2}(M,\mathbb{C}) \to \mathbb{R}$ is the \textit{momentum} corresponding to
$X$, which is defined as
\begin{equation}
	\label{eq:def-flux}
	\Phi_X(u)\coloneqq
	\frac{1}{2\pi}
	\int_{M} g\big(j(u),X\big) \,\mathrm{d}v_g,
\end{equation}
where $j(u)\colon M \to TM$ is the superfluid velocity (or superconducting
current), defined for $u(x) = u^1(x) + i u^2(x)$ as
\begin{equation}
	\label{eq:def-prejacobian}
	j(u)(x) \coloneqq
	u^1(x)\nabla u^2(x) - u^2(x)\nabla u^1(x).
\end{equation}
The vector field $j(u)$ is related to the \textit{vorticity} of $u$ which,
in the $3$--dimensional case, 
hence for $u \in W^{1,2}_{\mathrm{loc}}(\mathbb{R}^3,\mathbb{C})$,
is defined as $J(u):= \nabla u^1 \times \nabla u^2=\frac{1}2\mathrm{curl}(j(u))$.

For manifolds of general dimension we will handle the previous quantities using the language of differential forms, with the identification
$j(u)=u^1\mathrm{d}u^2-u^1\mathrm{d}u^2$ and 
$J(u):= \mathrm{d}u^1 \wedge \mathrm{d}u^2 \in \Omega^2(M)$, so that
\begin{equation}
	\label{eq:def-Jacobian}
	J(u) = \frac{1}{2} \,\mathrm{d} j(u).
\end{equation}
Rather than working
with the Gross-Pitaevskii energy, we will take the alternative variational
formulation of considering $E_\varepsilon$ restricted to the space
\begin{equation}
	\label{eq:def-Xphi}
	\mathcal{X}_{\phi} \coloneqq
	\Big\{
		u \in W^{1,2}(M,\mathbb{C}) :
		\Phi_X(u) = \phi,
	\Big\},
\end{equation}
which contains the complex-valued maps of fixed momentum $\phi>0$.
The restriction of $E_\varepsilon$ to $\mathcal{X}_\phi$ will be denoted as $E_{\varepsilon,\phi}$,
so that its critical points are solutions of~\eqref{eq:Ginzburg-PDE} for some
Lagrange multiplier $\lambda$ (see Lemma \ref{lemma:basic_functional} below).

The main contribution of this paper is to provide a lower bound on the number of 
solutions of \eqref{eq:Ginzburg-PDE}, according to the \textit{maximum velocity
set} of $X$ as defined in \eqref{eq:set_Sigma} below.
Such a lower bound is obtained for any pair of fixed vorticity and momentum
parameters which are \textit{sufficiently small}.
By \textit{solution} to \eqref{eq:Ginzburg-PDE} we mean a pair
$(\psi,\lambda)$, which, in analogy with the theory of Bose-Einstein
condensates, can be interpreted as the wave function $\psi$ of a condensate
with fixed momentum $\phi$ along with the scalar multiplier $\lambda$ which
defines the intensity and sense of the velocity field.
Moreover, our analysis also shows that the vorticity of such solutions is
concentrated around a small ring located close to a point of $M$ in which
$\|X\|$ is maximal.

We now give a precise formulation of our results.
Identifying the divergence-free vector field $X$ with a co-closed 1-form (still denoted by $X$), we make the stronger assumption that it is co-exact.
Moreover, it will also be assumed
that the maximum velocity set
\begin{equation}\label{eq:set_Sigma}
	\Sigma:= \big\{ x \in M : \norm{X(x)} = \max_{y \in M}\norm{X(y)} \big\},
\end{equation}
satisfies the following topological condition:
\begin{asu}[$\Sigma$ is a strong deformation retract of some tubular neighborhood]
	\label{ass:Sigmadelta-retraction}
	There exists $\delta>0$ such that $\Sigma$ is a strong deformation retract of its tubular neighborhood $\Sigma_\delta$, 
	where
	\begin{equation*}
		\Sigma_\delta := \big\{ x \in M : \mathrm{dist}_g(x,\Sigma) \le \delta \big\}.
	\end{equation*}
	In other words, there exists a continuous map
	\begin{equation*}
		h_\Sigma: [0,1] \times \Sigma_{\delta} \to \Sigma_\delta,
	\end{equation*}
	such that
	\begin{itemize}
		\item $h_\Sigma(0,x)=x$ for all $x \in \Sigma_\delta$;
		\item $h_\Sigma(1,x) \in \Sigma$ for all $x \in \Sigma_\delta$;
		\item $h_\Sigma(t,x)=x$ for all $(t,x) \in [0,1] \times \Sigma$.
	\end{itemize}
\end{asu}
We assume that the potential $W\colon \mathbb{C} \to [0,+\infty)$, which is of class $C^2$ and vanishes only on $\mathbb{S}^1$, 
satisfies the following standard conditions.
\begin{asu}[Nondegeneracy near $\mathbb{S}^1$]
	\label{ass:nondegeneracy}
	We have that
	\begin{equation*}
		\liminf_{\abs{u} \to 1} \frac{W(u)}{(1 - \abs{u})^2} > 0.
	\end{equation*}
\end{asu}
\begin{asu}[Subcritical growth assumption]\label{ass:subcritical}
	Let $p^*$ be the critical Sobolev exponent associated to
	$W^{1,2}(M,\mathbb{C})$, i.e., $p^*:= \frac{2N}{N-2}$. There exist
	$C_{\mathrm{sc},1}, C_{\mathrm{sc,2}}>0$ and $p \in [2,p^*)$ such that
	\begin{equation*}
		\sup_{u \in \mathbb{C}}\lvert D^2W(u) \rvert \leq C_{\mathrm{sc,1}}+C_{\mathrm{sc,2}}\lvert u\rvert^{p-2}.
	\end{equation*}
\end{asu}
\begin{asu}[Coercivity at infinity]\label{ass:coercivity}
	There exist $\alpha_c,R_c>0$ such that 
	\begin{equation*}
		\nabla W(u) \cdot u \geq \alpha_c\lvert u \rvert^2 \quad \mbox{ for all } u \in \mathbb{C} \mbox{ such that } \lvert u \rvert \geq R_c,
	\end{equation*}
	where $\nabla W(u)\cdot u$ denotes the inner product after identification with $\mathbb{R}^2$.
\end{asu}
Before stating the main result, let us recall that given a topological space $\mathfrak{X}$:
\begin{itemize}
	\item
		The \emph{Lusternik-Schnirelmann category} of $\mathfrak{X}$, 
		denoted by $\mathrm{cat}(\mathfrak{X})$, is the minimum number of 
		open contractible subsets required to cover $\mathfrak{X}$.
	\item
		The \emph{sum of all Betti numbers} of $\mathfrak{X}$ (with coefficients in $\mathbb{Z}_2$), is the evaluation 
		of the Poincaré polynomial of $\mathfrak{X}$ at $t = 1$, i.e.
		\[
			\mathcal{P}_1(\mathfrak{X}) = \sum_{k \in \mathbb{N}} \beta_k(\mathfrak{X}),
		\]
		where $\beta_k(\mathfrak{X})$ stands for the $k$-th Betti number of $\mathfrak{X}$.
\end{itemize}
\begin{theoremroman}[Main result]
	\label{theorem:main}
	Let $(M,g)$ be a closed Riemannian manifold
	of dimension $N \ge 3$,
	and $X\colon M \to TM $
	a smooth vector field such that
	its associated one-form 
	is co-exact	and~\ref{ass:Sigmadelta-retraction} holds.
	Let $W\colon \mathbb{C} \to [0,+\infty)$ be a $C^2$ potential with $W^{-1}(0)=\mathbb{S}^1$ that satisfies~\ref{ass:nondegeneracy}, \ref{ass:subcritical} and \ref{ass:coercivity}.
	Then, there exists $\phi^* > 0$,
	depending on $M, g$ and $X$,
	such that for every 
	$\phi \in (0,\phi^*)$
	there exists $\varepsilon^* = \varepsilon^*(\phi) >0$,
	depending also on $M, g, X$ and $W$,
	such that for every 
	$\varepsilon \in (0,\varepsilon^*)$
	the number of critical points of 
	$E_{\varepsilon,\phi}$ is at least $\mathrm{cat}(\Sigma)$.
	Moreover, if all the critical points of $E_{\varepsilon,\phi}$
	are non-degenerate, 
	the number of critical points of 
	$E_{\varepsilon,\phi}$ is at least 
	$\mathcal{P}_1(\Sigma)$.
\end{theoremroman}

\begin{remark}[Existence of minimizers]
	By using the direct method, one obtains the existence of one solution of \eqref{eq:Ginzburg-PDE} as the global minimizer of $E_{\varepsilon,\phi}$. Thus, we need to have $\mathrm{cat}(\Sigma) \geq 2$ in order for Theorem \ref{theorem:main} to be interesting.
\end{remark}

\begin{remark}[Further properties of the solutions]
	The solutions given by Theorem~\ref{theorem:main} are obtained in the proof as \textit{low energy} critical points of $E_{\varepsilon,\phi}$. The upper bound on their energies is given by the quantity $c(\phi)$ defined in \eqref{eq:def-sublevel} below. Moreover, our proof also shows that the energy densities of these are mostly concentrated near $\Sigma$, see Lemma \ref{lem:jacobian-concentration} below for a precise statement.
\end{remark}

\begin{remark}[On the dependence of the critical parameters on $W$]
	Notice that the critical momentum parameter $\phi^*$ in Theorem
	\ref{theorem:main} \textit{does not} depend on the choice of the potential
	$W$. However, the critical parameter $\varepsilon^*$ does
	(see Remark \ref{rem:dependence_eps} below).
\end{remark}

\begin{remark}[Existence of admissible vector fields for arbitrary manifolds]
	Given any closed smooth and oriented manifold $M$ with $\mathrm{dim}(M) \geq
	3$, it is always possible to find a vector field $X$ on $M$ which satisfies
	the assumptions of Theorem \ref{theorem:main}. In order to see this,
	consider $f$ an arbitrary smooth function on $M$ which is supported on a
	local chart on $M$. One can then use local coordinates to define the
	$(N-2)$-form
	\begin{equation*}
		Y:= f \,\mathrm{d}x_1 \wedge \ldots \wedge \mathrm{d}x_{N-2},
	\end{equation*}
	so that
	\begin{equation*}
		\mathrm{d}Y:= \frac{\partial f}{\partial x_{N-1}} \mathrm{d}x_1 \wedge \ldots \wedge \mathrm{d}x_{N-1}+\frac{\partial f}{\partial x_N} \mathrm{d}x_1 \wedge \ldots \wedge dx_{N-2} \wedge \mathrm{d}x_N.
	\end{equation*}
	Then, $X:= \star \mathrm{d}Y$ is co-exact, and moreover $\Sigma$ is the set of points in $M$ such that the function
	\begin{equation*}
		\left\lvert \frac{\partial f}{\partial x_{N-1}}\right\rvert^2+\left\lvert \frac{\partial f}{\partial x_{N}}\right\rvert^2
	\end{equation*}
	is maximal.
	Since $f$ is arbitrary, it can be chosen in a way such that $\Sigma$
	satisfies \ref{ass:Sigmadelta-retraction} and thus Theorem
	\ref{theorem:main} can be applied for $X$. Moreover, one can also consider
	suitable linear combinations of vectors fields as the one above in order to
	obtain new vector fields which satisfy the assumptions of Theorem
	\ref{theorem:main}.
	In particular, we deduce that given a positive integer $m$
	there exists $X_m$ which is co-exact and such that $\mathrm{cat}(\Sigma_m)
	\geq m$, where $\Sigma_m$ is defined as in \eqref{eq:set_Sigma} for $X_m$.
	Theorem \ref{theorem:main} then implies that the system
	\eqref{eq:Ginzburg-PDE} possesses at least $m$ variational solutions. 
\end{remark}

\begin{remark}[The case of manifolds with trivial first homology group]\label{remark:divergence_free}
	If $H_{1}(M)$, the first homology group of $M$ (with integer coefficients),
	is trivial, then then any smooth co-closed 1-form is co-exact.
    Recall also that $H_1(M)$ is trivial when $M$ is simply connected ($H_1(M)\cong \pi_1(M)^{\mathrm{ab}}
=0$).
    In this case one might find admissible vector fields considering $f=(f^1,\dots,f^n)$ supported in a local chart and setting
	\begin{equation*}
		\tilde{X}:= \sum_{j=1}^Nf^j\, \, \widehat{\mathrm{d}x_j},
	\end{equation*}
	where $\widehat{\mathrm{d}x_j}$ stands for the $(N-1)$-form
	\begin{equation*}
		\mathrm{d}x_1 \wedge \ldots \wedge \mathrm{d}x_{j-1}  \wedge \mathrm{d}x_{j+1} \wedge \ldots \wedge \mathrm{d}x_N
	\end{equation*}
	(with the obvious meaning when either $j=1$ or $j=N$). One then obtains
	\begin{equation*}
		\mathrm{d}\tilde{X}= \mathrm{div}(f)\mathrm{d}x_1 \wedge \ldots \wedge \mathrm{d}x_{N},
	\end{equation*}
	which vanishes provided that $f$ is divergence free.
	The vector field $X=\star \tilde{X}$ is then co-exact,
	and since $\Sigma$ is the set of
	points in which $\lvert f \rvert$ is maximal, one can clearly choose $f$ in
	a way such that \ref{ass:Sigmadelta-retraction} holds. One might also
	consider suitable linear combinations of vector fields as above in order to
	produce other admissible vector fields.
\end{remark}

\begin{remark}[On the case $\Sigma=M$]
	\label{rem:Sigma-equalto-M}
    If $X$ is a nonzero constant norm 1-form
    which satisfies the assumption of Theorem \ref{theorem:main}, then $\Sigma=M$ and one recovers a multiplicity result according to the topology of $M$ as in the works \cite{Benci1995,BenciCerami1994,benci2022-NLA,benci2024-corrigendum,corona2024multiplicity,andrade2024-JFA,andrade2024-multiplicityClusters}.
	However, such a 1-form might not exist for an arbitrary $M$ due to Poincaré-Hopf Theorem: a necessary condition is that the \textit{Euler characteristic} of $M$, denoted as $\chi(M)$, vanishes.
    Such a condition  on M might still not be sufficient, as we need our constant norm $1$--form to be co-exact. 
    To the best of the authors' knowledge, it is an open problem proving that a manifold supports a constant norm co-exact (or just co-closed) $1$--form knowing that it supports a constant norm $1$--form.
\end{remark}

\begin{remark}[On the non-degeneracy condition for solutions]
	In analogy with the Allen-Cahn setting of~\cite{benci2022-NLA,corona2024multiplicity}, one expects the critical points of $E_{\varepsilon,\phi}$ to be generically nondegenerate, the genericity being understood with respect to perturbations of the Riemannian metric $g$. We shall not address this issue in the present paper.
\end{remark}
The problem considered in the present paper as well as the methods used in the proof of Theorem \ref{theorem:main} are closely related to the $\Gamma$-convergence theory for the family of functionals $(E_\varepsilon)_{\varepsilon>0}$ as well as to certain isoperimetric-type problems in codimension 2.
In order to be more precise, let us begin by noticing that since $X$ is co-exact, i.e., there exists a $(N-2)$--form
$Y$ such that 
$\mathrm{d}Y = \star X \in \Omega^{N-1}(M)$,
the momentum $\Phi_X(u)$ of $u \in W^{1,2}(M,\mathbb{C})$ can be given in terms of the vorticity $J(u)$
(also called the \emph{Jacobian} of $u$; 
likewise, $j(u)$ is referred to as the \emph{pre-Jacobian}).
Indeed, as $M$ is a manifold without boundary,
by Stokes' Theorem we have:
\begin{multline}
	\label{eq:flux-Jacobian}
	\Phi_X(u) = 
	\frac{1}{2\pi} \int_M j(u)  \wedge \star X
	= \frac{1}{2\pi}\int_M j(u) \wedge \mathrm{d}Y \\
	= \frac{1}{\pi}\int_M J(u) \wedge Y 
	= \frac{1}{\pi} \int_M \sprod{J(u),\star Y}\mathrm{d}v_g.
\end{multline}
Following \cite{BOS2004}, we will also use the term \textit{flux} to refer to the quantity $\Phi_X(u)$.

At this point, it is worth noting that a nonzero momentum constraint forces the formation of singularities 
as the vortex core size parameter $\varepsilon$ tends to zero.
This stems from the fact $J(u)=0$ for every $u \in W^{1,2}(M,\mathbb{S}^1)$ (see~\cite{AlbertiBaldoOrlandiJEMS} and the references therein).
Thus, it readily follows by \eqref{eq:flux-Jacobian} that $\Phi_X(u)=0$ whenever $u \in W^{1,2}(M,\mathbb{S}^1)$.
As a consequence, when $\phi>0$ one cannot find a family
$(u_\varepsilon)_{\varepsilon>0}$ in $ \mathcal{X}_{\phi}$
whose unscaled energies remain bounded,
namely,
$
\sup_{\varepsilon>0} \lvert \log \varepsilon \rvert E_{\varepsilon,\phi}(u_\varepsilon)<+\infty
$.
Indeed, one readily checks that in such a case
we would reach a contradiction,
as after passing to the limit $\varepsilon \to 0$ we would obtain a map 
$u \in W^{1,2}(M,\mathbb{S}^1)$
such that $\Phi_X(u) = \phi \neq 0$.

The previous observation on the momentum is somewhat reminiscent of the \textit{topological obstructions} (such as those created by boundary data) usually found in the study of Ginzburg-Landau models departing from 
the seminal work \cite{BBH}, 
which focuses on the case of a bounded domain in the Euclidean plane.
Several extensions of \cite{BBH} (in particular to higher dimensional settings) were later considered, see~\cite{Alberti2001, AlbertiBaldoOrlandiJEMS, AlbertiBaldoOrlandi-Indiana2005, BethuelBrezisOrlandi2001, BOS2004, Jerrard1999, JerrardSoner2002, Sandier1998}.
In particular, in \cite{AlbertiBaldoOrlandi-Indiana2005} it is shown that the $\Gamma$-limit of the family of functionals $(E_{\varepsilon})_{\varepsilon>0}$ as $\varepsilon \to 0$ is a generalization of the area functional in codimension 2.
The results in \cite{AlbertiBaldoOrlandi-Indiana2005} rely on \cite{AlbertiBaldoOrlandiJEMS} and on the so-called ball construction established in independent works by Jerrard \cite{Jerrard1999} and Sandier \cite{Sandier1998}.
Following the work \cite{BOS2004} (which focuses on the Eucliden setting with constant vector field) one observes that when $(E_{\varepsilon})_{\varepsilon>0}$ is supplemented with the momentum constraint associated to \eqref{eq:def-flux} for $\phi>0$, one finds that the $\Gamma$-limit is the generalized area functional for ``submanifolds" (to be understood in the generalized sense of \textit{currents}) of codimension 2 which satisfy the geometric constraint (to be referred to as \textit{flux constraint} in the sequel) that their ``volume" weighted by $X$ must be equal to $\phi$.
In particular, as $\varepsilon \to 0$, the minimizers of $E_{\varepsilon,\phi}$ (more precisely, their vorticity/Jacobians) are close,
in a suitable sense that will be made precise 
later,
to the boundaries of oriented ``submanifolds" of $M$
that have flux $\phi$ with respect to $X$ 
and minimize the ``area" of their boundaries.
The latter objects are thus solutions to an \textit{isoperimetric}-type problem in codimension 2. 
Therefore, the momentum constraint associated to \eqref{eq:def-flux} can be thought of as a sort of higher-codimensional analog for our Ginzburg-Landau functional of the standard \textit{volume} (also called \textit{mass}) constraints often considered in Allen-Cahn settings.
Notice that the extension of the notion of \textit{volume constraint} from the standard isoperimetric problem to higher codimension is not straightforward nor unique, as in the latter case one does not have a natural notion of ``enclosed region" (that is, there are many regions of codimension 2 which share the same given boundary).
The notion of volume weighted by a $1$-form that we consider here was introduced by Salavessa \cite{Salavessa2010} and corresponds to the alternative $(2)$ in the subsequent work by Morgan and Salavessa \cite{MorganSalavessa} provided that $X$ is assumed to be \textit{co-exact}. When $X$ is not co-exact, the study of the corresponding isoperimetric problem becomes problematic and we do not even know if it is well-posed (see Remark \ref{rem:coexact} below for more details). This is the reason for which $X$ is assumed to be co-exact.
In analogy with the codimensional 1 case, one expects that when $\phi$ is
sufficiently small, the geometric problem on $M$ should converge to the
Euclidean one and thus that  our ``isoperimetric submanifolds'' of codimension
$2$ should be ``geodesic disks'' with normal vector close to $X$. Furthermore,
such small discs should be centered around a point in which the modulus of $X$
is maximal so that the perimeter is minimized. In this direction, a
contribution of this paper, possibly of independent interest, is to provide a
proof of the fact that
\textit{almost minimizers} of our isoperimetric problem have \textit{almost}
all of their volume concentrated around a geodesic disk close to a point where
$X$ maximizes its norm.

Let us also point out that, in contrast with isoperimetric problems in
codimension 1, their higher-dimensional analogues seem to have been much less
explored so far.
Moreover, to our knowledge, only the works \cite{MorganSalavessa,Salavessa2010}
consider the notion of volume weighted by a given vector field $X$ as we do
here. 
The classical paper by Almgren \cite{Almgren1986} and more recent works (e.g.,
Mazzeo, Pacard and Zolotareva \cite{Mazzeo2017}) take a different notion of
``enclosed volume" of higher codimension (namely, the minimal volume of all the
submanifolds which share the same boundary).
Flux-constrained Ginzburg-Landau-type problems and their link with higher-dimensional isoperimetric problems also seem to have remained quite unexplored so far: to our knowledge, besides the present paper, it appears only on \cite{BOS2004,chiron2004} as well as on different but related work by Román, Sandier and Serfaty \cite{Roman2019,RomanSandierSerfaty2023}.

We conclude this introduction with some additional bibliographical comments.
Multiplicity results similar in spirit to Theorem~\ref{theorem:main}
have been proved for several Allen–Cahn-type models with a volume constraint,
see~\cite{benci2020-calcvar,benci2022-NLA,benci2024-corrigendum,corona2024multiplicity,andrade2024-JFA,andrade2024-multiplicityClusters}.
As for Ginzburg-Landau-type models, multiplicity results in the planar setting where provided by Almeida and Bethuel \cite{AlmeidaBethuel} and later improved by Feng and Zhou \cite{ZhouZhou}.
In the work \cite{JerrardSternberg2009}, Jerrard and Sternberg showed that, in some situations, $\Gamma$-convergence can be used in order to prove the existence of (possibly unstable) critical points of the approximating functionals. In particular, they applied their result to the 3D Ginzburg-Landau functional with and without magnetic field.
More recently, existence results for Ginzburg-Landau functionals have been established in various settings using different techniques: the reader is referred to Colinet, Jerrard and Sternberg \cite{ColinetJerrardSternberg} and works by De Philippis, Pigati and Stern \cite{dephilippis-pigati,PigatiStern,Stern}.
We note that, in contrast with the present paper, the references we just cited deal with Ginzburg-Landau functionals \emph{without} a momentum/flux constraint and their approaches and methods are rather different than ours.

Let us now consider the time-dependent Gross-Pitaevskii equation in our manifold $M$, which reads
\begin{equation}\label{eq:GrossPitaevskii}\tag{$t$-GP$_\varepsilon$}
	i \, \lvert \log \varepsilon \rvert \Psi_t + \Delta \Psi = \frac{1}{\varepsilon^2}\nabla W(\Psi), \quad \Psi: M \times \mathbb{R} \to \mathbb{C}.
\end{equation}
As for the standard nonlinear Schrödinger equation on $M$ (see Mukherjee \cite{Mukherjee} and Taylor \cite{Taylor}) solutions of \eqref{eq:Ginzburg-PDE} can be interpreted as traveling wave solutions of \eqref{eq:GrossPitaevskii} provided that $X$ is a \textit{Killing vector field}.
Indeed, in such case $X$ generates a $1$-parameter group of isometries, denoted as $(\mathfrak{g}(t))_{t \in \mathbb{R}}$. Thus, if $u$ solves \eqref{eq:Ginzburg-PDE} for certain $\varepsilon$ and $\lambda$, then
\begin{equation*}
	\Psi: M \times \mathbb{R} \to \mathbb{C},
	\quad \Psi(x,t):= u(\mathfrak{g}(\lambda t)x) 
\end{equation*}
is a traveling solution to \eqref{eq:GrossPitaevskii} moving at speed $\lambda$ and with profile $u$.
However, when $X$ is Killing the functional $E_{\varepsilon,\phi}$ and the space $\mathcal{X}_{\phi}$ are invariant by the action of the group of isometries generated by $X$, so that one can trivially produce a continuum of critical points for $E_{\varepsilon,\phi}$ by applying the group action to the minimizer, unless the latter is invariant by the action.
Thus, the main result of this paper is not well suited for a situation in which $X$ is Killing.
In such a framework, one should instead aim at proving the existence of multiple \textit{critical orbits} for $E_{\varepsilon,\phi}$ (that is, multiple orbits of critical points, the orbits being understood as those generated by the action of the isometry group on $\mathcal{X}_\phi$), which in turn would lead to the existence of multiple \textit{geometrically distinct} traveling waves for \eqref{eq:GrossPitaevskii}.
It is tempting to believe the previous fact might be proven by combining the methods of this paper with those of \textit{equivariant critical point theory}, that is, critical point theory adapted to variational problems which are invariant by the action of a group (cf. Bartsch~\cite{BartschBook}).

\subsection*{Organization of the paper.}
After this introduction, Section~\ref{sec:notation} introduces the main notation and preliminary results,
recalling also standard notions from Geometric Measure Theory.
Section~\ref{sec:outline-of-proof} gives a detailed outline of the proof of
Theorem \ref{theorem:main}, which is based on the  \emph{photography method}
described in Theorem~\ref{theorem:abstract-photography}.
In Section~\ref{sec:PS} we prove the Palais-Smale condition for
$E_{\varepsilon,\phi}$ as well as other functional properties, 
while in Section~\ref{sec:gamma-convergence} we provide some results related to
$\Gamma$-convergence for the flux-constrained energy functional in closed
Riemannian manifolds.
Sections~\ref{sec:photography} and~\ref{sec:barycenter}
are devoted to the explicit construction of 
the two maps required by the photography method,
namely the photography and the barycenter, respectively.
In Section~\ref{sec:finalproof}, 
we prove that the composition 
of these two maps is homotopic to the identity and complete the proof of
Theorem~\ref{theorem:main}.

\subsection*{Acknowledgments}
This work was partially supported Italian National Institute of High Mathematics (INdAM). 
This work was supported by the Gruppo Nazionale per l’Analisi Matematica, la Probabilità e le loro Applicazioni (GNAMPA) – INdAM.\\
D.Corona was supported by INdAM as ``titolare di una borsa per l’estero dell’Istituto Nazionale di Alta Matematica'' and FAPESP \#2022/16097-2.\\
S. Nardulli was supported by
FAPESP Auxílio Jovem Pesquisador \#2021/05256-0,
CNPq Bolsa de Produtividade em Pesquisa 1D \#12327/2021-8, 23/08246-0,
``Geometric Variational Problems in
Smooth and Nonsmooth Metric Spaces'' \#441922/2023-6.\\
R. Oliver-Bonafoux was supported by Program Horizon Europe Marie
Sklodowska-Curie Post-Doctoral Fellowship (HORIZON-MSCA-2023-PF-01). Grant agreement:	101149877. Project acronym: NFROGS.\\
We thank Giacomo Canevari for useful discussions related to this work.

\section{Notation and Preliminaries}
\label{sec:notation}

\subsection{General notions of Geometric Measure Theory}

We first introduce some basic notation and conventions from Geometric Measure Theory,
recalling some definitions and properties relative to currents on a Riemannian manifold $M$.
Classical references on the subject are the books by Federer \cite{Federer1996} and Simon \cite{Simon1983}.

For any $k\in\{0,\dots,N\}$, let $\Omega^k(M)$ be the space of smooth 
$k$--forms on $M$,
and let $\star\colon \Omega^k(M)\to\Omega^{N-k}(M)$
be the Hodge star operator associated with the metric $g$.
We will implicitly identify $k$--vectors with their associated $k$--forms, and vice versa.
We denote by $\mathcal{D}_k(M)$ the space of $k$--currents on $M$,
and by $\mathcal{B}_k(M)$
the space of $k$--currents that are boundaries, hence
\[
	\mathcal{B}_k(M)\coloneqq \big\{\partial T:\ T\in \mathcal{D}_{k+1}(M)\big\}\subset \mathcal{D}_k(M).
\]
We denote by $\mathcal{I}_k(M)\subset \mathcal{D}_k(M)$ the space of integral
$k$--currents, and by
\[
	\mathcal{IB}_k(M)\coloneqq \big\{\partial T:\ T\in \mathcal{I}_{k+1}(M)\big\}
\]
the space of $k$--dimensional integral boundaries on $M$.
Given an open set $U\subset M$ and a $k$--dimensional integral current $T$ on $M$,
we denote by $T|_U$ the restriction of $T$ to $U$.
Then $T|_U$ is again an $k$--dimensional integral current on $M$.

For any open set $U\subset M$, the mass of $T$ in $U$ is defined by
\[
	\mathbf{M}_U(T)\coloneqq
	\sup\big\{\, T(\omega):
		\ \omega\in \Omega^k(M),\ 
		\|\omega\|_\infty\le 1,\ 
	\mathrm{spt}(\omega)\subset U \big\},
\]
and we denote by $\|T\|$ the mass measure of $T$,
so that $\mathbf{M}_U(T)=\|T\|(U)$.
We simply write $\mathbf{M}$ instead of $\mathbf{M}_M$,
so $\mathbf{M}(T)=\|T\|(M)$.
For any $T \in \mathcal{D}_k(M)$, its flat norm in $U$ is
\[
	\mathbf{F}_U(T)\coloneqq
	\inf\big\{\, \mathbf{M}_U(R)+\mathbf{M}_U(\tilde{R}):
		\ T=R+\partial \tilde{R},\
		R\in\mathcal{D}_k(M),\
	\tilde{R}\in\mathcal{D}_{k+1}(M)\big\}.
\]
The following compactness result will be used at several instances:
if $(T_n)_{n \in \mathbb{N}}\subset \mathcal{I}_{N-1}(M)$ and
$\sup_{n \in \mathbb{N}}\big(\mathbf{M}(T_n)+\mathbf{M}(\partial T_n)\big) <+\infty$, then there exists $T \in \mathcal{I}_{N-1}(M)$ such that $\mathbf{F}(T_n-T) \to 0$ as $n \to \infty$.

For any Lipschitz map $\varphi\colon U \subset M \to \mathbb{R}^N$ and
$T \in \mathcal{D}_k(U)$,
the pushforward of $T$ by $\varphi$, denoted $\varphi_{\#}T$, is defined as
\[
	\varphi_{\#}T(\alpha)\coloneqq T(\varphi^{\#}\alpha),
	\qquad \forall\, \alpha \in \Omega^{k}(\mathbb{R}^N),
\]
where $\varphi^{\#}\alpha \in \Omega^k(U)$ denotes the pullback of the differential form $\alpha$.
One readily verifies that $\partial(\varphi_{\#}T)=\varphi_{\#}(\partial T)$.
Moreover, if $T$ is an integral current on $U$,
then $\varphi_{\#}T$ is an integral current on $\mathbb{R}^N$.

\subsection{Our Setting}

As stated in the introduction, 
throughout this paper, $(M,g)$ 
denotes a closed (i.e., compact and without boundary) 
Riemannian manifold of dimension $N \ge 3$.  
By compactness, we can fix a constant $r_0$ such that  
\begin{equation}
	\label{eq:def-r0}
	0 < r_0 < \frac{1}{2}\, \mathrm{inj}(M),
\end{equation}
where $\mathrm{inj}(M)$ is the injectivity radius of $M$,
and such that every geodesic ball of radius $r < r_0$
is strongly convex.

We assume the existence of a fixed, nonzero vector field $X\colon M \to TM$ whose associated one-form
is co-exact, 
so that there exists $Y \in \Omega^{N-2}(M)$
such that 
\[
	\mathrm{d}Y = \star X \in \Omega^{N-1}(M).
\]
As a consequence, $X$ is also divergence-free, 
namely
\[
	\mathrm{div} X(x) = 0, \qquad \forall x \in M.
\]
Since $X$ is not the zero vector field, we can assume,
without loss of generality, that
\begin{equation}
	\label{eq:maxXone}
	\max_{x \in M} \|X(x)\|  = 1.
\end{equation}

As $X$ is co-exact, Stokes' Theorem for currents implies that
\begin{equation}\label{eq:flux_coexact}
	T(\star X)=\partial T(Y), \quad \mbox{ for all } T \in \mathcal{I}_{N-1}(M),
\end{equation}
which means that the flux of a current depends only on its boundary.
Given $\phi>0$, we set
\[
	\mathcal{I}^\phi_{N-1}(M) \coloneqq \Big\{T \in \mathcal{I}_{N-1}(M): T(\star X ) = \phi \Big\}.
\]
and
\[
	\mathcal{IB}_{N-2}^\phi(M)
	\coloneqq \{ S \in \mathcal{B}_{N-2}(M): S(Y)=\phi\}.
\]
Notice that \eqref{eq:flux_coexact} implies that if $T \in \mathcal{I}_{N-1}^\phi(M)$ then $\partial T \in \mathcal{IB}_{N-2}^\phi(M)$.

Finally, we recall that the set $\Sigma \subset M$ where $\|X\|$ attains its maximum (see~\eqref{eq:set_Sigma}) is a strong deformation retract of some tubular neighborhood $\Sigma_\delta \subset M$,
as required by
Assumption~\ref{ass:Sigmadelta-retraction}.

We consistently use superscripts
to denote the components of complex numbers or vectors, 
while subscripts indicate elements of a sequence.  
Hence, we express any function $u \in W^{1,2}(M,\mathbb{C})$ as
$ u(x) = u^1(x) + i u^2(x) $,
and its gradient belongs to the complexified tangent space 
$ T^*M \otimes \mathbb{C} $, so that
\[
	\nabla u(x) = \nabla u^1(x) + i \nabla u^2(x).
\]
For any function $u \in W^{1,2}(M,\mathbb{C})$, 
its \emph{pre-Jacobian}, denoted by $ j(u): M \to TM $, 
is the vector field defined as in~\eqref{eq:def-prejacobian},  
which we recall here for convenience:
\[
	j(u) \coloneqq 
	\mathrm{Im}\big(u(x) \cdot \overline{\nabla u(x)}\big)
	=
	u^1(x)\nabla u^2(x) - u^2(x)\nabla u^1(x).
\]
We also recall  the definition (rescaled) \emph{flux functional} with respect to $ X $
(see~\eqref{eq:def-flux}), 
denoted by $ \Phi_X\colon W^{1,2}(M,\mathbb{C}) \to \mathbb{R} $, as
\[
	\Phi_X(u) = \int_{M} g\big(j(u),X\big)\,\mathrm{d}v_g.
\]

The $\Gamma$-convergence of Ginzburg-Landau functionals is formulated in terms of the \textit{Jacobian} of maps from $M$ to $\mathbb{C}$, which, as stated in \eqref{eq:def-Jacobian}, are defined as
\begin{equation*}
	J(u)\coloneqq \frac{1}{2} \,\mathrm{d} j(u) \in \Omega^2(M), \quad \mbox{ for all } u \in W^{1,2}(M,\mathbb{C}),
\end{equation*}
so that, in any local chart of $ M $, 
we have
\begin{equation}
	\label{eq:JacobianFormula}
	J(u)  = \sum_{1 \leq i < j \leq N}
	\left(
		\frac{\partial u^1}{\partial x^i}
		\frac{\partial u^2}{\partial x^j}
		-
		\frac{\partial u^2}{\partial x^i}
		\frac{\partial u^1}{\partial x^j}
	\right)
	\mathrm{d}x^i \wedge \mathrm{d}x^j.
\end{equation}
Through the Hodge star operator $\star$ given by the metric $g$, the Jacobian of $u$
can be identified with the $(N-2)$--form
$\star J(u)$, and then with a $(N-2)$--dimensional vector field on $M$.
As a consequence, through the integration operator,
$\star J(u)$ identifies an $(N-2)$--dimensional current on $M$.
More precisely, once $\star J(u)$ is seen as an 
$(N-2)$--dimensional vector field,
for any $\omega \in \Omega^{N-2}(M)$
we have
\[
	\star J(u) (\omega) = \int_M 
	\sprod{\star J(u),\omega}_g\mathrm{d}x.
\]
We now introduce the isoperimetric problem in codimension two
associated with the vector field $X$:
\begin{equation}\label{eq:inf_JM}
	J_M(\phi, X) \coloneqq
	\inf\left\{
		\norm{\partial T}(M) : T \in \mathcal{I}^\phi_{N-1}
	\right\}.
\end{equation}
In other words, the problem consists in minimizing the mass of the boundary
among all integral $(N-1)$--dimensional currents with flux $\phi$
with respect to $X$.
For small values of $\phi$,
$J_M(\phi, X)$ behaves as in the Euclidean case
for a constant vector field (see Proposition~\ref{prop:iso_expansion}); hence,
\begin{equation}
	\label{eq:J_M-asymptotic}
	J_M(\phi, X)
	= \gamma_{N-1}\,\phi^{\frac{N-2}{N-1}}
	+ o\big(\phi^{\frac{N-2}{N-1}}\big),
\end{equation}
where $\gamma_{N-1} > 0$ depends only on the dimension
and coincides with the constant appearing in the isoperimetric inequality
for $(N-2)$--dimensional integral boundaries in $\mathbb{R}^N$
(cf.~Almgren~\cite{Almgren1986}).

\begin{remark}[On $J_M(\phi,X)$ and the co-exactness assumption on $X$]\label{rem:coexact}
	It is not difficult to see that the infimum in \eqref{eq:inf_JM} is attained for all $\phi>0$. This follows from the fact that the flux depends only on the boundary of the current 
	(by~\eqref{eq:flux_coexact}) since $X$ is co-exact. Indeed, given $(T_n)_{n \in \mathbb{N}}$ a minimizing sequence for \eqref{eq:inf_JM}, one can apply the classical compactness result for integral currents to $(\partial T_n)_{n \in \mathbb{N}}$ along with \eqref{eq:flux_coexact} to get the existence of a minimum for $J_M(\phi,X)$. The previous argument cannot be implemented when $\star X$ is not exact, as then the flux does not depend only on the boundary of currents. As a consequence, in such case one needs to have compactness for the minimizing sequence itself (rather than the sequence of boundaries). However, the existence of a minimizing sequence with uniformly bounded mass (which would imply compactness) is not obvious to us (the difficulty seems to lie in the fact that the flux constraint must be respected) and thus we do not know whether $J_M(\phi,X)$ is attained when $\star X$ is no longer exact. Similarly, the proof of the asymptotic identity \eqref{eq:J_M-asymptotic} also lies on the fact that $\star X$ is exact. To sum up, we do not see at this point how our analysis can be carried out for the case in which $X$ is merely assumed to be co-closed.
\end{remark}

\section{Outline of the proof}
\label{sec:outline-of-proof}

The proof of Theorem~\ref{theorem:main}
is based on the so-called \emph{photography method},
a technique in critical point theory introduced by
Benci and Cerami~\cite{BenciCerami1994}
(see also~\cite[Section IV.7]{Benci1995}). 

We recall that given a $C^1$ functional $E$ on a smooth Banach manifold $\mathfrak{M}$ and a sequence $(x_n)_{n \in \mathbb{N}} \subset \mathfrak{M}$, the latter is said to be a Palais-Smale sequence for $E$ when $(E(x_n))_{n \in \mathbb{N}}$ is bounded and $(\norm{\mathrm{d} E(x_n)}_{(T_{x_n}\mathfrak{M})^*})_{n \in \mathbb{N}}$ tends to zero as $n$ tends to $0$.
In this setting, the functional  is said to satisfy the Palais-Smale condition when every Palais-Smale sequence possesses a strongly convergent subsequence.
\begin{theorem}[Photography method, cf.~\cite{Benci1995,BenciCerami1994}]
	\label{theorem:abstract-photography}
	Let $\mathfrak{X}$ be a topological space,
	$\mathfrak{M}$ be a $C^2$-Hilbert manifold,
	and let $E\colon \mathfrak{M}\to\mathbb{R}$ be a $C^1$-functional.
	For any $c\in\mathbb{R}$, define the sublevel set
	\[
		E^c\coloneqq\{x\in \mathfrak{M} : E(x)\leq c\}.
	\]
	Assume the following conditions hold:
	\begin{enumerate}
		\item
			$\inf_{x \in \mathfrak{M}}E(x)>-\infty$;
		\item
			$E$ satisfies the Palais--Smale (PS) condition;
		\item
			There exist two continuous maps
			$f \colon \mathfrak{X}\to E^c$ and
			$\beta\colon E^c\to \mathfrak{X}$
			such that $\beta\circ f$ is homotopic
			to the identity map of $\mathfrak{X}$.
	\end{enumerate} 
	Then, the number of critical points of $E$ in $E^c$ 
	is at least $\mathrm{cat}(\mathfrak{X})$.
	Moreover, if $\mathfrak{M}$ is contractible and $\mathrm{cat}(\mathfrak{X})>1$,
	then there exists at least one additional
	critical point of $E$ outside $E^c$.
	Furthermore, if the number of critical points is finite,
	there exists $c_0\in(c,\infty)$ such that 
	$E^{c}$ contains exactly $\mathcal{P}_1(\mathfrak{X})$ critical points,
	and if $X$ is contractible, then   
	$E^{c_0}\setminus E^{c}$
	contains $\mathcal{P}_1(\mathfrak{X})-1$ critical points,
	counted with multiplicity.
\end{theorem}

The name ``photography'' reflects that
$f(\mathfrak{X})$ may be regarded as faithful picture of $\mathfrak{X}$
inside $\mathfrak{M}$.
Indeed,
the retraction
$\beta\colon E^c\to Z$ with $\beta\circ f\simeq \mathrm{id}_\mathfrak{X}$,
one can recognize $\mathfrak{X}$ inside the sublevel $E^c$;
in particular, homotopy-invariant quantities of $\mathfrak{X}$
(such as $\mathrm{cat}(\mathfrak{X})$, cup-length, the Betti numbers)
transfer to $E^c$.
This passage of topology is precisely what underlies the lower bounds on the number of critical points.

Beyond the previously mentioned applications of the photography method 
in the study of the Allen-Cahn functional
(see~\cite{benci2020-calcvar,benci2022-NLA,benci2024-corrigendum,corona2024multiplicity,andrade2024-JFA,andrade2024-multiplicityClusters}),
this technique has also been employed in various other contexts.  
For instance, Cingolani and Lazzo~\cite{MR1646619,MR1734531} 
used it to analyze standing waves of a nonlinear Schr\"odinger equation.  
More recently, Petean~\cite{MR3912791}
applied it to establish a multiplicity result 
for the Yamabe equation,
while Alarcón, Petean and Rey~\cite{MR4761862} 
investigated its role in the study of conformal metrics
with constant $Q$-curvature.

\medskip
As our functional 
$E_{\varepsilon,\phi}$ is
clearly bounded from below,
and the Palais--Smale condition for it
can be readily checked (see Section~\ref{sec:PS}),
to apply Theorem~\ref{theorem:abstract-photography}
in our setting we just need to define two maps: 
a \emph{photography map} 
$f_{\varepsilon,\phi}\colon \Sigma \to \mathcal{X}_{\phi}$, 
which captures the topology of $\Sigma$ within the space of admissible functions, 
and a \emph{barycenter map}
$\beta\colon E_{\varepsilon,\phi}^{c} \to \Sigma$, 
which allows us to construct a homotopy to the identity. 
The interplay between these two maps is the key ingredient of the proof.

Thanks to the $\Gamma$--convergence results
for the flux--constrained Ginzburg--Landau functional,
one expects that, for $\phi$
and $\varepsilon$ sufficiently small, the minimizers
of $E_{\varepsilon,\phi}$
are close to tiny vortex rings around 
the boundaries of $(N-2)$--dimensional geodesic disks
orthogonal to $X$.
Moreover, these disks are centered at points in $\Sigma$,
where $\|X\|$ attains its maximum,
so that they have the smallest possible radius
(and hence the smallest boundary area)
compatible with the given flux constraint.
Therefore, we construct the photography map
$f_{\varepsilon,\phi}\colon \Sigma \to \mathcal{X}_\phi$
as an element of the recovery sequence (as $\varepsilon \to 0$)
associated with the boundaries of
\[
	D_X(p,r(p,\phi)) = 
	\big\{\exp_p v \in M:
		v \in T_pM,\, g(v,X(p)) = 0,\,
	\|v\| \le r(p,\phi)\big\},
\]
where $r(p,\phi)$ is chosen so that 
the flux of $X$ through the disk equals $\phi$.
Thanks to $\Gamma$--convergence 
and the estimate of the boundary area
of small disks,
we obtain a function $c:(0,+\infty) \to \mathbb{R}$
such that 
\begin{equation}
	\label{eq:informal-photo-bound}
	\lim_{\varepsilon \to 0}
	E_{\varepsilon,\phi}\big(f_{\varepsilon,\phi}(p)\big)
	= \|\partial D_X(p,r(p,\phi))\|(M)
	\le c(\phi),
	\qquad \forall p \in \Sigma.
\end{equation}
In other words, 
we obtain an estimate for the smallest sublevel
that contains the image of the photography map,
since
\[
	f_{\varepsilon,\phi}(\Sigma) \subset E_{\varepsilon,\phi}^{c(\phi)}.
\]
Most importantly, this upper bound coincides, 
at leading order, with the
``isoperimetric'' functional in codimension~$2$,
namely
\[
	c(\phi) = J_M(\phi,X) + O(\phi),
	\qquad \mbox{as } \phi \to 0.
\]

Based on the above estimate,
we can define the
barycenter map on this sublevel, provided
the flux parameter $\phi$ is sufficiently small.
Indeed, from the analysis of the isoperimetric problem
in higher codimension, we know that its
\textit{almost minimizing} sequences are given by currents
whose mass is almost entirely concentrated in a single ball 
centered near $\Sigma$.
Since this result is of independent interest,
we postpone its statement and
proof to Appendix~\ref{app:generalized-compactness}.
Therefore, 
using again the $\Gamma$--convergence
(in particular, the equi-coerciveness and liminf properties),
we can transfer this concentration property 
to the energy densities of functions in $E_{\varepsilon,\phi}^{c(\phi)}$.
Then, by means of an intrinsic barycenter map,
we construct a continuous map 
$\beta^*\colon E_{\varepsilon,\phi}^{c(\phi)} \to M$
such that
\[
	\beta^*\big(E_{\varepsilon,\phi}^{c(\phi)}\big) \subset \Sigma_{\delta},
\]
and hence a map $\beta$ with values in $\Sigma$ by composing $\beta^*$
with the retraction $h(1,\cdot)\colon\Sigma_\delta \to \Sigma$
given by~\ref{ass:Sigmadelta-retraction}.
Finally, the homotopy to the identity of $\beta \circ f_{\varepsilon,\phi}$
is obtained through a careful use of the exponential map.

To the best of our knowledge,
this is the first multiplicity result for the
Ginzburg--Landau functional in this spirit,
as previous results are devoted to
Allen--Cahn functionals
(see, e.g., \cite{benci2022-NLA,benci2024-corrigendum,corona2024multiplicity,andrade2024-JFA,andrade2024-multiplicityClusters}).
Beyond the technical difficulties stemming from the higher codimension, 
the key difference here is the presence of a distinguished set $\Sigma\subset M$.
Indeed, for the Allen--Cahn functional, the ``natural'' constraint is the $L^1$ mass of the functions, and the natural photography map is taken as an element of the recovery sequence for a geodesic ball of full dimension and prescribed volume.
Since the volume is not weighted by any scalar field,
it is the same at every point of the manifold;
accordingly, the natural photography map is defined on
the whole $M$.
By contrast, in the flux--constrained Ginzburg--Landau
setting the relevant cost depends on the vector field $X$:
for small flux $\phi$, the optimal
vortex rings concentrate near points where $\|X\|$ is maximal,
i.e., on $\Sigma$.
Hence the photography map is naturally defined on $\Sigma$,
not on all of $M$.
Moreover, the situation $\Sigma=M$ is exceptional:
it would require the existence of a nowhere-vanishing vector field $X$ of constant norm whose assosiated one-form is co-exact,
a condition typically precluded
by topological obstructions
(cf. Remark~\ref{rem:Sigma-equalto-M}).

\section{Functional setting}\label{sec:PS}
The results provided in this section show that $E_{\varepsilon,\phi}$ is a $C^1$ functional which satisfies the Palais-Smale condition and that its critical points give rise to classical solutions of \eqref{eq:Ginzburg-PDE}. Such results are standard, but, to the best of our knowledge, their proofs cannot be found in the previous literature. Therefore, we include them here for completeness.

Let $p$ be the subcritical exponent given by \ref{ass:subcritical} and $C_{p}$ be the positive quantity associated to the compact embedding $W^{1,2}(M,\mathbb{C}) \hookrightarrow L^{p}(M,\mathbb{C})$. In particular, we have
\begin{equation}\label{eq:embedding}
	\norm{u}_{L^{p}(M,\mathbb{C})} \leq C_p\norm{u}_{W^{1,2}(M,\mathbb{C})} \quad \mbox{ for all } u \in W^{1,2}(M,\mathbb{C}).
\end{equation}
\begin{lemma}[Basic inequalities for the potential]\label{lemma:basic_potential}
	For all $u,v$ and $w$ in $W^{1,2}(M,\mathbb{C})$ we have
	\begin{equation}\label{eq:integral_potential_1}
		\int_M \lvert D^2W(u)w \cdot v \rvert \mathrm{d}v_g \leq \left(C_{\mathrm{sc,1}}+C_{\mathrm{sc,2}}\norm{u}_{L^{p}(M,\mathbb{C})}\right)\norm{w}_{L^{p}(M,\mathbb{C})}\norm{v}_{L^{p}(M,\mathbb{C})},
	\end{equation}
	and
	\begin{align}\label{eq:integral_potential_2}
		\int_M &\lvert (\nabla W(u)-\nabla W(w)) \cdot v \rvert \mathrm{d}v_g \nonumber \\ &\leq \left(C_{\mathrm{sc,1}}+C_\mathrm{sc,2}\left(\norm{u}_{L^{p}(M,\mathbb{C})}+\norm{w}_{L^{p}(M,\mathbb{C})}\right)\right)\norm{u-w}_{L^p(M,\mathbb{C})}\norm{v}_{L^{p}(M,\mathbb{C})}.
	\end{align}
\end{lemma}
\begin{proof}
	From \ref{ass:subcritical}, we obtain
	\begin{equation}\label{eq:integral_potential_0}
		\int_M \lvert D^2W(u)w \cdot v \rvert \mathrm{d}v_g \leq \int_M \left( C_{\mathrm{sc,1}}+C_{\mathrm{sc,2}}\lvert u \rvert^{p-2}\right)\lvert w \rvert \lvert v \rvert \mathrm{d}v_g.
	\end{equation}
	By applying Hölder's inequality (for three functions instead of two) to 
	\begin{equation*}
		\int_M \lvert u \rvert^{p-2}\lvert w \rvert \lvert v \rvert \mathrm{d}v_g
	\end{equation*}
	with the exponent $p/(p-2)$ for the first factor and $p$ for each of the other two factors (so that the sum of the inverses of the exponents is equal to 1), we obtain \eqref{eq:integral_potential_1}
	Subsequently, by Fubini's Theorem we notice that
	\begin{equation*}
		\int_M \lvert (\nabla W(u)-\nabla W(w)) \cdot v \rvert \mathrm{d}v_g \leq \int_0^1 \int_M \lvert D^2W(su+(1-s)w) \rvert \lvert u-w \rvert \lvert v \rvert \mathrm{d}v_g \mathrm{d}s
	\end{equation*}
	and \eqref{eq:integral_potential_2} then follows by applying \eqref{eq:integral_potential_1}.
\end{proof}
\begin{lemma}[Basic properties of the functional setting]\label{lemma:basic_functional}
	The following properties hold:
	\begin{enumerate}[label=\textup{\roman*)}]
		\item\label{item:C1} For any positive $\varepsilon$, the functionals $E_{\varepsilon}$ and $\Phi_X$ are $C^1$ on $W^{1,2}(M,\mathbb{C})$ and
			\begin{equation}\label{eq:differential_energy}
				\mathrm{d}E_\varepsilon(u)(v)=\frac{1}{\pi \lvert \log \varepsilon \rvert}\int_{M}\left(g(\nabla u,\nabla v)+\frac{1}{\varepsilon^2}\nabla W(u) \cdot v \right) \mathrm{d}v_g,
			\end{equation}
			\begin{equation}\label{eq:differential_flux}
				\mathrm{d}\Phi_X(u)(v)=\frac{1}{4\pi}\int_M g(v^1\nabla u^2-v^2\nabla u^1,X) \mathrm{d}v_g,
			\end{equation}
			for any $u$ and $v$ in $W^{1,2}(M,\mathbb{C})$.
		\item For any positive $\phi$, the set $\mathcal{X}_\phi$ is a Banach manifold.\label{item:manifold}
		\item For any positive $\varepsilon$ and $\phi$, if $u$ is a critical point of $E_{\varepsilon,\phi}$ then there exists a real number $\lambda$ such that the pair $(u,\lambda)$ solves \eqref{eq:Ginzburg-PDE} in the classical sense.\label{item:regularity}
	\end{enumerate}
\end{lemma}
\begin{proof}
	\ref{item:C1}
	The fact that $\Phi_X$ is of class $C^1$ with differential as in \eqref{eq:differential_flux} follows by direct computations. One now checks that by \ref{ass:subcritical}, a Taylor expansion and \eqref{eq:embedding}, the functional $E_{\varepsilon}$ is finite on $W^{1,2}(M,\mathbb{C})$. Subsequently, by \eqref{eq:integral_potential_1} in Lemma \ref{lemma:basic_potential} along with \eqref{eq:embedding} one obtains that $E_\varepsilon$ is differentiable in the Fréchet sense on $W^{1,2}(M,\mathbb{C})$ with differential as in \eqref{eq:differential_energy}. The fact that $E_{\varepsilon}$ is of class $C^1$ follows then by \eqref{eq:integral_potential_2} in Lemma \ref{lemma:basic_potential} and \eqref{eq:embedding}.

	\ref{item:manifold}
	Notice then that for all $u$ in $\mathcal{X}_\phi$ one has $\mathrm{d}\Phi_X(u)(u)=\frac{\phi}{2}$, which is different than zero. This shows that $\mathrm{d}\Phi_X(u)$ is nonzero. As a consequence, the thesis follows by the Implicit Function Theorem.

	\ref{item:regularity}
	The Lagrange multiplier theorem implies that if $u$ is a critical point of $E_{\varepsilon,\phi}$, then there exists a Lagrange multiplier $\lambda \in \mathbb{R}$ such that $u$ is a weak solution of \eqref{eq:Ginzburg-PDE}. The fact that $u$ is a classical solution follows then by a standard combination of assumption \ref{ass:subcritical}, Sobolev embeddings, elliptic regularity and a bootstrap argument.
\end{proof}
\begin{lemma}[The Palais-Smale condition]\label{lemma:PS}
	For any positive $\varepsilon$ and $\phi$, the functional $E_{\varepsilon,\phi}$ satisfies the Palais-Smale condition.
\end{lemma}
\begin{proof}
	Let $(u_n)_{n \in \mathbb{N}}$ be a Palais-Smale sequence for $E_{\varepsilon,\phi}$, which means that $(E_{\varepsilon,\phi}(u_n))_{n \in \mathbb{N}}$ is bounded and there exists $(\lambda_n)_{n \in \mathbb{N}}$ a sequence of real numbers such that
	\begin{equation}\label{eq:PS_derivative}
		\lim_{n \to \infty} \mathrm{d}E_{\varepsilon}(u_n)-\lambda_n \mathrm{d}\Phi_X(u_n)=0, \quad \mbox{ on } W^{-1,2}(M,\mathbb{C}).
	\end{equation}
	Since $(E_{\varepsilon,\phi}(u_n))_{n \in \mathbb{N}}$ is bounded, it readily follows by the coercivity assumption \ref{ass:coercivity}
	that $(u_n)_{n \in \mathbb{N}}$ is a bounded sequence in $W^{1,2}(M,\mathbb{C})$.
	As a consequence, we obtain from \eqref{eq:PS_derivative} 
	\begin{equation}\label{eq:PS_derivative_2}
		\lim_{n \to \infty} \mathrm{d}E_{\varepsilon}(u_n)(u_n)-\lambda_n \mathrm{d}\Phi_X(u_n)(u_n)=0.
	\end{equation}
	According to \eqref{eq:differential_flux} in Lemma \ref{lemma:basic_functional}, one has
	\begin{equation}\label{eq:flux_unun}
		\mathrm{d}\Phi_X(u_n)(u_n)= \frac{\phi}{2}, \quad \mbox{ for all } n \in \mathbb{N}.
	\end{equation}
	By combining \eqref{eq:integral_potential_2} in Lemma \ref{lemma:basic_potential} with \eqref{eq:differential_energy} in Lemma \ref{lemma:basic_functional} and the fact that $(u_n)_{n \in \mathbb{N}}$ is bounded in $W^{1,2}(M,\mathbb{C})$, it follows
	\begin{equation}\label{eq:sup_derivatives_unun}
		\sup_{n \in \mathbb{N}}\lvert \mathrm{d}E_\varepsilon(u_n)(u_n) \rvert <+\infty.
	\end{equation}
	By combining \eqref{eq:PS_derivative_2}, \eqref{eq:flux_unun} and \eqref{eq:sup_derivatives_unun}, we get that $(\lambda_n)_{n \in \mathbb{N}}$ is bounded, thus it converges to some $\lambda_\infty \in \mathbb{R}$ up to an extraction. Since $(u_n)_{n \in \mathbb{N}}$ is bounded in $W^{1,2}(M,\mathbb{C})$, it converges weakly in $W^{1,2}(M,\mathbb{C})$ to some $u_\infty$. Up to a further extraction, we also have that $(u_n)_{n \in \mathbb{N}}$ converges strongly to $u_\infty$ in $L^p(M,\mathbb{C})$. These facts imply that
	\begin{equation}\label{eq:PS_lim_flux}
		\lim_{n \to \infty}\mathrm{d}\Phi_X(u_n)(v)=\mathrm{d}\Phi_X(u_\infty)(v), \quad \mbox{ for all } v \in W^{1,2}(M,\mathbb{C}),
	\end{equation}
	\begin{equation}\label{eq:PS_lim_potential}
		\lim_{n \to \infty}\int_{M}\nabla W(u_n) \cdot v\mathrm{d}v_g =\int_M \nabla W(u_\infty) \cdot v \mathrm{d}v_g,\quad \mbox{ for all } v \in W^{1,2}(M,\mathbb{C}),
	\end{equation}
	and
	\begin{equation}\label{eq:PS_lim_gradient}
		\lim_{n \to \infty}\int_{M}g(\nabla u_n,\nabla v)\mathrm{d}v_g=\int_{M}g(\nabla u_\infty,\nabla v)\mathrm{d}v_g, \quad \mbox{ for all } v \in W^{1,2}(M,\mathbb{C}),
	\end{equation}
	where for \eqref{eq:PS_lim_potential} we have also used  \eqref{eq:integral_potential_2} in Lemma \ref{lemma:basic_potential}. Combining \eqref{eq:PS_derivative} with \eqref{eq:PS_lim_flux}, \eqref{eq:PS_lim_potential} and \eqref{eq:PS_lim_gradient} we obtain
	\begin{equation}\label{eq:PS_eq_limit}
		\mathrm{d}E_{\varepsilon}(u_\infty)-\lambda_\infty \mathrm{d}\Phi_X(u_\infty)=0,  \quad \mbox{ on } W^{-1,2}(M,\mathbb{C}).
	\end{equation}
	Notice now that the $L^p$-strong convergence of $(u_n)_{n \in \mathbb{N}}$ and the $L^2$-weak convergence of $(\nabla u_n)_{n \in \mathbb{N}}$ imply that $(j(u_n))_{n \in \mathbb{N}}$ converges weakly in $L^2(M,\mathbb{C})$ to $j(u_\infty)$, which means that $u_\infty \in \mathcal{X}_\phi$ and thus, by \eqref{eq:PS_eq_limit}, it is a critical point of $E_{\varepsilon,\phi}$. Moreover, it also follows that
	\begin{equation}\label{eq:PS_lim_flux_n}
		\lim_{n \to \infty}\mathrm{d}\Phi_X(u_n)(u_n)=\mathrm{d}\Phi_X(u_\infty)(u_\infty).
	\end{equation}
	Using again \eqref{eq:integral_potential_2} in Lemma \ref{lemma:basic_potential}, we get
	\begin{equation}\label{eq:PS_lim_potential_n}
		\lim_{n \to \infty}\int_{M}\nabla W(u_n) \cdot u_n \mathrm{d}v_g =\int_M \nabla W(u_\infty) \cdot u_\infty \mathrm{d}v_g.
	\end{equation}
	By plugging \eqref{eq:PS_lim_flux_n} and \eqref{eq:PS_lim_potential_n} into \eqref{eq:PS_derivative_2} and then using \eqref{eq:PS_eq_limit}, we get
	\begin{equation}\label{eq:PS_lim_gradient_n}
		\lim_{n \to \infty}\int_{M}g(\nabla u_n,\nabla u_n)\mathrm{d}v_g=\int_{M}g(\nabla u_\infty,\nabla u_\infty)\mathrm{d}v_g.
	\end{equation}
	Since $(\nabla u_n)_{n \in \mathbb{N}}$ converges $L^2$-weakly to $\nabla u_\infty$, we deduce from \eqref{eq:PS_lim_gradient_n} that $(u_n)_{n \in \mathbb{N}}$ converges to $u_\infty$ strongly in $W^{1,2}(M,\mathbb{C})$. This completes the proof.
\end{proof}
\section{\texorpdfstring{$\Gamma$}{}--convergence}
\label{sec:gamma-convergence}

In this section we show that, 
for any flux constraint $\phi$,
the family of functionals
$(E_{\varepsilon,\phi})_{\varepsilon > 0}$
$\Gamma$--converges to the 
mass functional for the $(N-2)$--dimensional integral boundaries in $M$, as $\varepsilon \to 0$.
Actually, we will present and prove just the liminf and 
equi--coerciveness properties of the $\Gamma$--convergence,
as we do not need the limsup property in its full generality.
Indeed, we will construct
recovery sequences only for 
boundaries of small $(N-1)$--dimensional 
disks (cf. Proposition~\ref{lem:photoConstruction}).
Therefore, we prove the following
two results.

\begin{proposition}[Equicoercivity]
	\label{prop:Gamma-equicoer}
	For any family $(u_\varepsilon)_{\varepsilon > 0}\subset \mathcal{X}_\phi$
	such that
	$\limsup_\varepsilon E_\varepsilon(u_\varepsilon) < + \infty$,
	there exists a $(N-2)$--dimensional
	integral boundary $\partial T$ such that 
	$T \in \mathcal{I}_{N-1}^\phi$ and
	$ \star J(u_\varepsilon)$ converges in the flat norm
	to $\pi\partial T$,
	up to subsequences.
\end{proposition}
\begin{proposition}[Liminf property]
	\label{prop:Gamma-liminf}
	For any family $(u_\varepsilon)_{\varepsilon > 0}\subset \mathcal{X}_\phi$
	and $\partial T \subset \mathcal{IB}_{N-2}^\phi$
	such that
	$ \star J(u_\varepsilon)$ converges in the flat norm
	to $\pi\partial T$,
	we have that 
	\begin{equation}
		\label{eq:Gamma-liminf}
		\liminf_{\varepsilon \to 0}
		E_{\varepsilon,\phi}(u_\varepsilon) \ge
		\norm{\partial T}(M);
	\end{equation}
	more generally,
	for any open subset $U \subset M$
	we have
	\begin{equation}
		\label{eq:Gamma-liminfLoc}
		\liminf_{\varepsilon \to 0}
		\frac{1}{\pi \abs{\log \varepsilon}}
		\int_U 
		\Big(\frac{1}{2}\norm{\nabla u_\varepsilon}^2
			+ \frac{1}{\varepsilon^2}
		W(u_\varepsilon)\Big)\mathrm{d}v_g
		\ge 
		\norm{\partial T}(U).
	\end{equation}
\end{proposition}

The $\Gamma$–convergence of the Ginzburg–Landau functional
was first proved in the Euclidean
setting in~\cite{AlbertiBaldoOrlandi-Indiana2005}, building upon the so-called ``ball construction'' introduced in independent works by Jerrard \cite{Jerrard1999} and Sandier \cite{Sandier1998}.
Its adaptation to closed Riemannian manifolds is well known and widely used (see, e.g., \cite{ColinetJerrardSternberg,dephilippis-pigati}),
although a complete proof is often omitted or only sketched.
To the best of our knowledge,
the only explicitly written proof in the Riemannian
setting is \cite[Theorem~5.1]{MesaricThesis} in dimension three.
That argument actually extends to any dimension via a standard partition of unity argument that 
allows to invoke the Euclidean $\Gamma$--convergence
results 
(as already suggested in~\cite{AlbertiBaldoOrlandi-Indiana2005},
there for compact manifolds with boundary).
For completeness, we provide self–contained
proofs of Propositions~\ref{prop:Gamma-equicoer} 
and~\ref{prop:Gamma-liminf}.

\begin{proof}[Proof of Proposition~\ref{prop:Gamma-equicoer}]

	Let $r_0$ be as in~\eqref{eq:def-r0},
	so that $r_0 \le \operatorname{inj}(M)/2$,
	and $\{x_1,\dots, x_n\} \subset M$ a finite collection of points 
	such that $\bigcup_{i = 1}^n B_g(x_i,r_0) = M$.
	As $r_0 < \operatorname{inj}(M)/2$, for any 
	$B_g(x_i,r_0)$ there exists a local chart that covers it,
	hence a bi-Lipschitz map
	$\varphi_i \colon B_g(x_i,r_0) \to U_i \subset \mathbb{R}^N$.
	For any $i = 1,\dots, n$, let us define
	$v_{\varepsilon,i} \in W^{1,2}(U_i,\mathbb{C})$
	as
	\[
		v_{\varepsilon,i}   = u_\varepsilon \circ \varphi_i^{-1}.
	\]
	By standard estimates of the Riemannian metric 
	with respect to the Euclidean one, 
	which can be applied by considering a smaller radius
	$r < r_0$ if necessary,
	there exists a positive constant $C$
	such that 
	\[
		E_\varepsilon(v_{\varepsilon_i})
		= \frac{1}{\pi|\log\varepsilon|}
		\int_{\textstyle U_i} 
		\left(\frac{1}{2}\lVert \nabla v_{\varepsilon,i}\rVert^2
		+ \frac{1}{\varepsilon^2} W(v_{\varepsilon,i})\right)
		\mathrm{d}x
		\le C \, E_{\varepsilon,\phi}(u_\varepsilon),
		\quad \forall i = 1,\dots,n.
	\]
	By hypothesis, this implies that
	\[
		\limsup_{\varepsilon\to 0} E_{\varepsilon}(v_{\varepsilon,i})
		< C \limsup_{\varepsilon\to 0} E_{\varepsilon,\phi}(u_{\varepsilon}) 
		< +\infty,
		\qquad \forall i = 1,\dots, n,
	\]
	and we can invoke the Euclidean $\Gamma$--convergence 
	equicoercivity result given in~\cite[Theorem 1.1~(i)]{AlbertiBaldoOrlandi-Indiana2005}.
	Therefore, after a diagonalization extraction, 
	we obtain an infinitesimal sequence $(\varepsilon_j)_{j \in \mathbb{N}}$ in $(0,1)$
	and integral boundaries $S_i \in \mathcal{IB}_{N-2}(U_i)$
	for $i=1,\dots,n$
	such that 
	\[
		\lim_{j \to \infty}
		\mathbf{F}_{U_i}\big(
			\star J(v_{\varepsilon_j,i})
			- \pi \, S_i
		\big) = 0,
		\qquad \forall i= 1,\dots, n,
	\]
	where $\mathbf{F}_{U_i}$ denotes the flat norm operator
	for currents in $U_i \subset \mathbb{R}^N$.
	Then, 
	since every map $\varphi_i$ is bi-Lipschitz, 
	we obtain also
	\begin{equation}
		\label{eq:equicoercivity-proof1}
		\lim_{j \to \infty}
		\mathbf{F}_{B_g(x_i,r)}\big(
			\star J (u_{\varepsilon_j})
			- \pi (\varphi_i^{-1})_\# S_i 
		\big) 
		\le 
		\lim_{j \to \infty}
		C\, \mathbf{F}_{U_i}\big(
			\star J(v_{\varepsilon_j,i})
			- \pi \, S_i
		\big) = 0, 
	\end{equation}
	for some positive constant $C$.
	Now, let $(\chi_i)_{i = 1,\dots,n}$
	be a partition of unity subordinate to 
	$\big(B_g(x_i,r)\big)_{i = 1,\dots,n}$,
	and define the current $S \in \mathcal{D}_{N-2}(M)$ as
	\[
		S = \sum_{i= 1}^n
		\chi_i \, (\varphi_i^{-1})_\sharp S_i.
	\]
	We then obtain
	\begin{align*}
		\mathbf{F}_{M}\big(\star J(u_{\varepsilon_j}) - \pi S\big)
	&\le 
	\sum_{i=1}^n
	\mathbf{F}_{M}\big(\chi_i\big(\star J(u_{\varepsilon_j}) - \pi S\big)\big) \\
	&\le 
	\sum_{i=1}^n
	\mathbf{F}_{B_g(x_i,r)}\!\big(\star J(u_{\varepsilon_j}) - \pi
	(\varphi_i^{-1})_{\#} S_i\big).
	\end{align*} 

	Then, by using~\eqref{eq:equicoercivity-proof1} we obtain
	\[
		\lim_{j \to \infty}
		\mathbf{F}_{M}\big(
			\star J(u_{\varepsilon_j})-
			\pi S
		\big) = 0.
	\]
	At this point, one sees that the current $S$ is an integral boundary,
	i.e., there exists $T \in \mathcal{I}_{N-1}(M)$ such that
	$\partial T=S$.
	Indeed, by the definition of the flat norm, there exist currents
	$Q^{j}$ and $R^{j}$ with
	$\star J(u_{\varepsilon_j})-\pi S = Q^{j}+\partial R^{j}$ and
	$\mathbf{M}(Q^{j})+\mathbf{M}(R^{j})\to0$.
	Since $\mathbf{M}(\star J(u_{\varepsilon_j}))$ is uniformly bounded, the isoperimetric
	inequality on $M$ provides fillings $T_j\in\mathcal{I}_{N-1}(M)$ with
	$\partial T_j=\star J(u_{\varepsilon_j})$ and $\sup_j \mathbf{M}(T_j)<\infty$.
	Hence
	\[
		S = \partial T_j - Q^{j} - \partial R^{j}
		= \partial(T_j - R^{j}) - Q^{j}.
	\]
	By a small–cycle filling estimate , write
	$Q^{j}=\partial \widetilde R^{j}$ with $M(\widetilde R^{j})\le C\,M(Q^{j})\to0$.
	Therefore
	\[
		S = \partial\big(T_j - R^{j} - \widetilde R^{j}\big),
		\qquad
		\sup_j M\big(T_j - R^{j} - \widetilde R^{j}\big)<\infty.
	\]
	By compactness of integral currents, up to a subsequence
	$T_j - R^{j} - \widetilde R^{j}\rightharpoonup T\in\mathcal I_{N-1}(M)$,
	and $\partial T= S$.

	Finally,
	as $J(u_\varepsilon)$ converges to $\pi \partial T$
	in the flat norm,
	then the flux constraint is preserved in the limit.
	Indeed, as $X$ is co–exact,
	we have
	\[
		\phi = \frac{1}{\pi}\int_M J(u_\varepsilon)\wedge Y
		\to \partial T(Y)=T(\mathrm{d}Y)=T(\star X),
	\]
	and the limit current $T$
	belongs to $\mathcal{I}_{N-1}^\phi$.
\end{proof}
\begin{proof}[Proof of Proposition~\ref{prop:Gamma-liminf}]
	As in the proof of equicoercivity, we apply a localization argument to invoke the Euclidean $\Gamma$--convergence counterpart.
	In particular, we rely on~\cite[Theorem 5.2]{JerrardSoner2002}.
	Although
	\cite{JerrardSoner2002} treats the standard potential $W(u)=\tfrac14(1-|u|^2)^2$,
	the same $\Gamma$–convergence lower bound holds
	for any potential $W$ fulfilling our hypotheses 
	(cf.~\cite{AlbertiBaldoOrlandi-Indiana2005}).

	For any $x\in M$ and $r\in(0,r_0)$, let
	$\varphi\colon B_g(x,r)\to U\subset\mathbb{R}^N$
	be a bi-Lipschitz chart, and set
	$v_\varepsilon \coloneqq u_\varepsilon\circ\varphi^{-1}\in W^{1,2}(U,\mathbb{C})$
	and $\nu \coloneqq \varphi_\#(\partial T)\in \mathcal{B}_{N-2}(U)$.
	By definition,
	\[
		\|\nu\|(A) = \|\partial T\|(\varphi^{-1}(A))
		\qquad\text{for every Borel set } A\subset U.
	\]
	Since $\varphi$ is bi-Lipschitz, flat convergence is preserved, namely
	\[
		\lim_{\varepsilon \to 0} \mathbf F_U(\star J v_\varepsilon - \pi\nu)= 0.
	\]
	Moreover, on $U$ the flat norm controls the $(C^{0,\alpha})^*$–norm of currents;
	hence
	\[
		J v_\varepsilon \to \pi\nu
		\quad\text{in } (C^{0,\alpha}_c(U))^*\text{ for every }\alpha\in(0,1).
	\]

	To apply~\cite[Theorem~5.2]{JerrardSoner2002},
	for each $\varepsilon$
	we define the energy measure of $v_{\varepsilon}$ on the Borel sets of $U$ by
	\[
		\mu_{\varepsilon}(A)
		\coloneqq  
		\frac{1}{\pi|\log\varepsilon|}
		\int_{\textstyle A}
		\Big(
			\frac{1}{2}
			\lVert\nabla v_{\varepsilon}\rVert ^2
			+ \frac{1}{\varepsilon^2}
		W(v_\varepsilon)\Big)
		\mathrm{d}x,
	\]
	Without loss of generality, 
	we may assume $\liminf E_{\varepsilon,\phi}(u_\varepsilon) < +\infty$.
	Then
	\begin{multline*}
		\liminf_{\varepsilon \to 0}
		\mu_{\varepsilon}(U)
		\le C_{\varphi}\, 
		\liminf_{\varepsilon \to 0}
		\frac{1}{\pi|\log\varepsilon|}
		\int_{\textstyle B_g(x,r)}
		\Big(
			\frac{1}{2}
			\lVert\nabla u_{\varepsilon}\rVert^2
			+ \frac{1}{\varepsilon^2}
		W(\varepsilon)\Big)
		\mathrm{d}v_g \\
		\le 
		C_{\varphi} 
		\liminf_{\varepsilon \to 0}
		E_{\varepsilon,\phi}(u_{\varepsilon})
		< +\infty,
	\end{multline*}
	where $C_{\varphi} > 0$ is a constant depending only
	on the bi-Lipschitz map $\varphi$.
	Moreover, $C_{\varphi} \to 1 $ as the radius 
	$r$ of the ball $B_g(x,r)$ tends to $0$.
	Since $\liminf_{\varepsilon \to 0}\mu_{\varepsilon}(U)$
	is bounded, there exists an
	infinitesimal sequence $(\varepsilon_j)_{j \in \mathbb{N}}$ in $(0,1)$ and a 
	Radon measure $\mu$ on $U$
	such that $\mu_{\varepsilon_j} \rightharpoonup \mu$
	in the sense of measures,
	as $j \to \infty$.

	By~\cite[Theorem~5.2]{JerrardSoner2002},
	we obtain that 
	\[
		|\nu| \ll \mu,
		\quad\text{and}\quad 
		\frac{\mathrm{d}|\nu|}{\mathrm{d}\mu} \le 1,
	\]
	hence
	\begin{multline}
		\label{eq:gammaliminf-proof1}
		\|\partial T\|\big(B_g(x,r)\big)
		= 
		\|\nu\|(U)
		\le  \mu(U)
		\le \liminf_{\varepsilon_j \to 0}
		\mu_{\varepsilon_j}(U) \\
		\le  
		C_{\varphi}\, 
		\liminf_{\varepsilon_j \to 0}
		\frac{1}{\pi|\log\varepsilon_j|}
		\int_{\textstyle B_g(x,r)}
		\Big(
			\frac{1}{2}
			\lVert\nabla u_{\varepsilon_j}\rVert^2
			+ \frac{1}{\varepsilon_j^2}
		W(u_{\varepsilon_j})\Big)
		\mathrm{d}v_g,
	\end{multline}
	so the liminf property holds locally.

	To obtain the final estimate,
	namely
	$\liminf_{\varepsilon \to 0}E_{\varepsilon,\phi}(u_{\varepsilon}) \ge \|\partial T\|(M)$,
	we can apply the above construction on
	a finite number of local charts.
	For any $\delta \in (0,1)$, choose points
	$\{x_1,\dots,x_n\}\subset M$ and a radius $r>0$
	such that the balls $B_g(x_i,r)$ are disjoint,
	the constants $C_{\varphi_i}$ associated with
	the bi-Lipschitz maps $\varphi_i:B_g(x_i,r)\to U_i$
	satisfy $|1-C_{\varphi_i}|\le\delta$ and
	\[
		\sum_{i = 1}^n
		\|\nu_i\|(U_i) = 
		\sum_{i = 1}^n
		\|\partial T\|(B_g(x_i,r))
		\ge \|\partial T\|(M) - \delta.
	\]
	Applying~\eqref{eq:gammaliminf-proof1}
	on each chart yields
	\begin{align*}
		\liminf_{\varepsilon \to 0}
		E_{\varepsilon,\phi}(u_{\varepsilon})
		  &
		  \ge
		  \sum_{i = 1}^n        
		  \liminf_{\varepsilon \to 0}
		  \frac{1}{\pi|\log\varepsilon|}
		  \int_{\textstyle B_g(x_i,r)}
		  \Big(
			  \frac{1}{2}
			  \lVert\nabla u_{\varepsilon}\rVert^2
			  + \frac{1}{\varepsilon^2}
		  W(u_\varepsilon)\Big)
		  \mathrm{d}v_g \\
		  & \ge
		  (1-\delta)\,
		  \sum_{i = 1}^n        
		  \liminf_{\varepsilon \to 0}
		  \frac{1}{\pi|\log\varepsilon|}
		  \int_{\textstyle B_g(x_i,r)}
		  \Big(
			  \frac{1}{2}
			  \lVert\nabla u_{\varepsilon}\rVert^2
			  + \frac{1}{\varepsilon^2}
		  W(u_\varepsilon)\Big)
		  \mathrm{d}v_g
		  \\
		  & \ge (1-\delta) \sum_{i = 1}^n \|\nu_i\|(U_i)
		  \ge 
		  (1- \delta) \|\partial T\|(M) - \delta.
	\end{align*}
	Letting $\delta \to 0$ gives~\eqref{eq:Gamma-liminf}.
	The same argument applies to any open subset $U\subset M$,
	which proves \eqref{eq:Gamma-liminfLoc}.
\end{proof}

\section{Photography map}
\label{sec:photography}

In this section, we construct the
photography map $f_{\varepsilon,\phi}\colon \Sigma \to \mathcal{X}_{\phi}$
and we estimate the smallest sublevel set
of $E_{\varepsilon,\phi}$ that contains its entire image.

For any $p \in \Sigma$, we denote by $D_X(p,r) \subset M$ 
the geodesic $(N-1)$–dimensional open disk 
centered at $p$, orthogonal to $X(p)$,
and with radius $r\in (0,r_0)$,
where we recall that $r_0$ has been fixed and it is 
less than (half of) the injectivity radius of $M$.
More precisely, we define
\begin{equation}
	\label{eq:def-SigmaX}
	D_X(p,r) \coloneqq \left\{
		\exp_{p} v \in M
		: v \in T_pM,\ g(v, X(p)) = 0,
	\ \norm{v} < r \right\}.
\end{equation}
\begin{figure}[ht]
	\centering
	\begin{overpic}[width=0.6\linewidth]{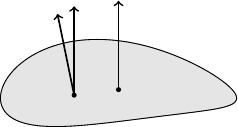}
		\put(53,16){$p$}
		\put(53,48){$X(p)$}
		\put(35,15){$q$}
		\put(35,48){$X(q)$}
		\put(14,45){$\nu(q)$}
		\put(-17,20){$D_X(p,r)$}
		\put(80,40){$M$}
	\end{overpic}
	\caption{A small geodesic disk $D_X(p,r)$ orthogonal to $X$ in a $3$--dimensional ambient manifold $M$.}
	\label{fig:smalldisk}
\end{figure}
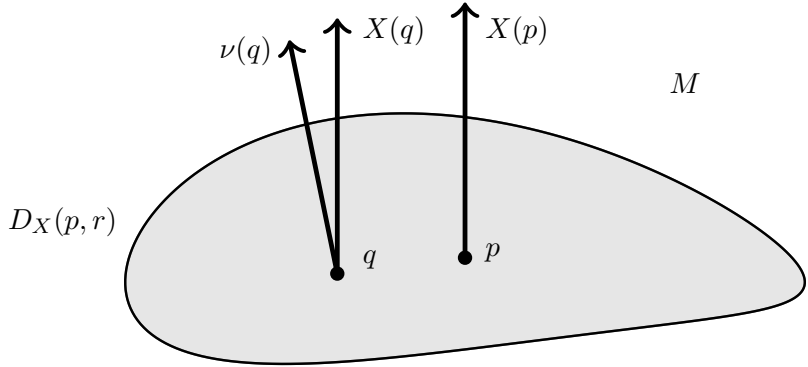
This construction is illustrated in Figure~\ref{fig:smalldisk}.

\begin{lemma}
	\label{lem:rightOrthoDisk}
	There exists $\phi_0 > 0$ such that
	for any $\phi \in [0,\phi_0]$
	and any $p \in \Sigma$, there exists a unique radius $r(p,\phi) \in [0,r_0]$ 
	such that the flux of $X$
	through the geodesic disk $D_X(p,r(p,\phi))$ 
	is equal to $\phi$, i.e.,
	\begin{equation}
		\label{eq:fluxCondition}
		\llbracket D_X(p,r(p,\phi)) \rrbracket (\star X)
		=
		\int_{D_X(p,r(p,\phi))} g(X,\nu) \, \mathrm{d}\sigma = \phi,
	\end{equation}
	where $\nu_p(x)$ is the continuous unit vector
	field orthogonal to 
	$D_X(p,r(p,\phi))$ such that $\nu_p(p) = X(p)$ 
	$($cf.~\eqref{eq:maxXone}$)$,
	and $\mathrm{d}\sigma$ denotes the induced volume form on 
	$D_X(p,r(p,\phi))$.
	Moreover, the function 
	$r \colon  \Sigma \times [0,\phi_0] \to [0,r_0]$
	is continuous and 
	there exists a constant $C_X > 0$
	such that 
	\begin{equation}
		\label{eq:rpphi-estimate}
		\frac{1}{C_X}
		\phi^{1/N-1}
		\le r(p,\phi) \le 
		C_X\,
		\phi^{1/N-1},
		\qquad \forall p \in \Sigma.
	\end{equation}
\end{lemma}
\begin{proof}
	By construction, for any $p \in \Sigma$
	we have $g(\nu(p),X(p)) = \norm{X(p)} = 1$;
	since $D_X(p,r)$ is a smooth submanifold for any $r \in (0,r_0)$
	$\nu$ and $X$ is continuous by hypothesis,
	there exists $\delta_p \in (0,r_0)$
	such that
	\[
		g\big(\nu_p(x),X(x)\big) \ge \frac{\norm{X(p)}}{2}
		\ge \frac{1}{2},
		\qquad \forall x \in D_X(p,\delta_p),
	\]
	so the map $r \mapsto \llbracket D_X(p,r) \rrbracket(\star X)$
	is strictly increasing in $(0,\delta_p)$.
	Indeed, by using also the compactness of $\Sigma$,
	there exists a dimensional constant $\omega_{N-2} > 0$
	such that
	\begin{equation}
		\label{eq:rightOrthoDisk-proof1}
		\frac{\mathrm{d}}{\mathrm{d}r} 
		\big(\llbracket D_X(p,r)\rrbracket(\star X)\big)
		\ge  \frac{\omega_{N-2}}{2} r^{N-2},
		\qquad \text{for a.e. } r \in [0,\delta_p].
	\end{equation}
	Since $\Sigma$ is a compact set, let us fix $\delta = \min_{p \in \Sigma} \delta_p > 0$,
	and let
	\[
		\phi_0 = \frac{1}{2}\min_{p \in \Sigma}
		\llbracket D_X(p,\delta) \rrbracket(\star X).
	\]
	By construction, for any $\phi \in [0,\phi_0]$
	and for any $p \in \Sigma$ there exists 
	a unique
	$r(p,\phi) \in [0,\delta]$
	such that~\eqref{eq:fluxCondition} holds.
	As all the construction involves only continuous functions,
	such a function $r\colon \Sigma \times [0,\phi_0] \to [0,r_0]$
	is continuous.
	Finally, integrating~\eqref{eq:rightOrthoDisk-proof1}
	and using again the compactness of $M$ and $\Sigma$,
	we obtain the existence of a constant $C_X > 0$
	such that~\eqref{eq:rpphi-estimate} holds.
\end{proof}

Through Lemma~\ref{lem:rightOrthoDisk} and
the limsup property of the $\Gamma$-convergence of the 
Ginzburg-Landau functional,
for any $\phi \in (0,\phi_0)$
and $\varepsilon > 0$, we define
the photography map $f_{\varepsilon,\phi}(p)$
as an element of the recovery sequence of the 
$(N-2)$-dimensional current $\partial D_X(p,r(p,\phi))$.
We begin with the following intermediate step,
where we construct the recovery sequence
without enforcing the flux constraint,
which is satisfied only in the limit.
\begin{lemma}
	\label{lem:photoConstruction-hat}
	For any $\phi \in (0,\phi_0)$, there exists 
	$\varepsilon_0 = \varepsilon_0(\phi) > 0$
	and a one-parameter 
	family of continuous functions 
	\[		
		(\hat{f}_{\varepsilon,\phi})_{\varepsilon \in (0,\varepsilon_0)} \colon \Sigma \to W^{1,2}(M,\mathbb{C})
	\]
	such that
	\begin{equation}
		\label{eq:photo-almost-flux}
		\lim_{\varepsilon \to 0} \Phi_X\big(\hat{f}_{\varepsilon,\phi}(p)\big)= \phi,
	\end{equation}
	and 
	\begin{equation}
		\label{eq:limEnergyPhoto-hat}
		\lim_{\varepsilon \to 0} E_{\varepsilon,\phi}(\hat{f}_{\varepsilon,\phi}(p))
		= \mathcal{H}_g^{N-2}\big(\partial D_X(p,r(p,\phi))\big).
	\end{equation}
\end{lemma}
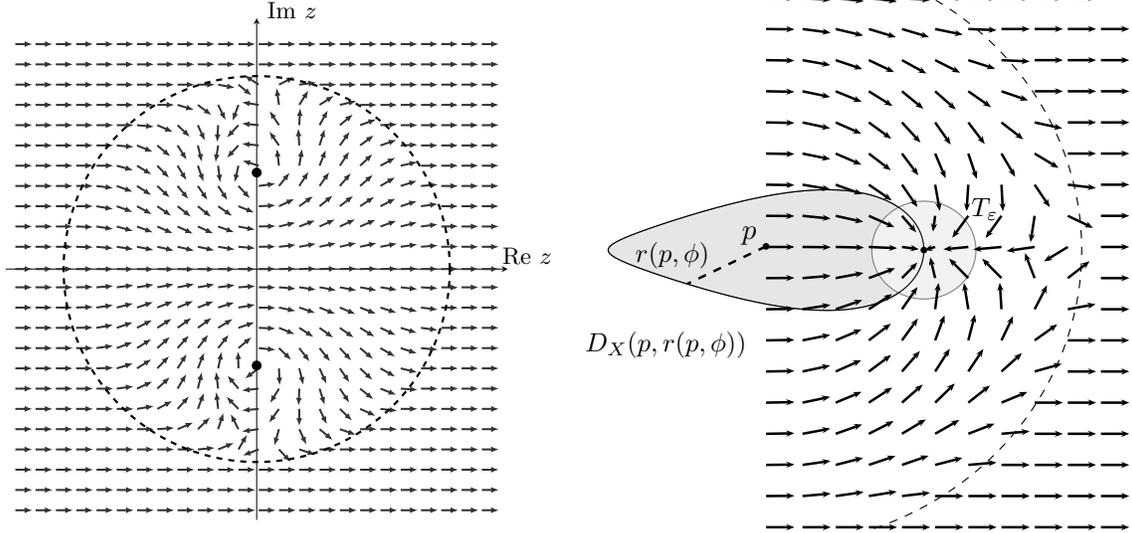
\begin{figure}[ht]
	\centering
	\begin{subfigure}[t]{0.46\textwidth}
		\centering
		\begin{overpic}[width=\linewidth]{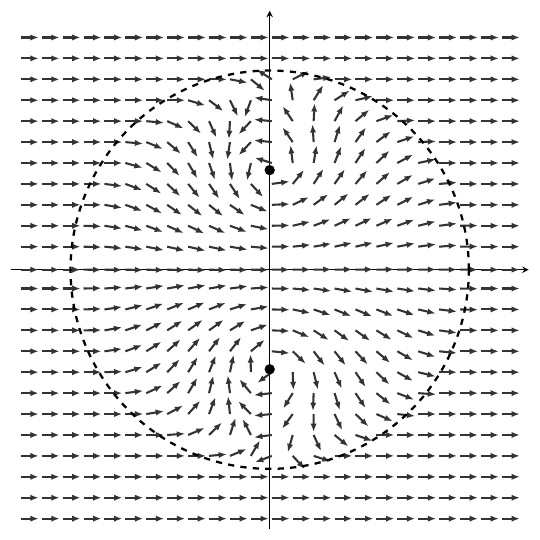}
			\put(97,51){\footnotesize $\operatorname{Re}\,z$} 
			\put(52,98){\footnotesize $\operatorname{Im}\,z$} 
		\end{overpic}
		\caption{The dipole
			vortex map $\omega_{p,\phi}$ given in~\eqref{eq:def-w_R},
		where the singular points are at $(0,\pm r(p,\phi))$.}
		\label{fig:photo1}
	\end{subfigure}\hfill
	\begin{subfigure}[t]{0.46\textwidth}
		\centering
		\begin{overpic}[width=\linewidth]{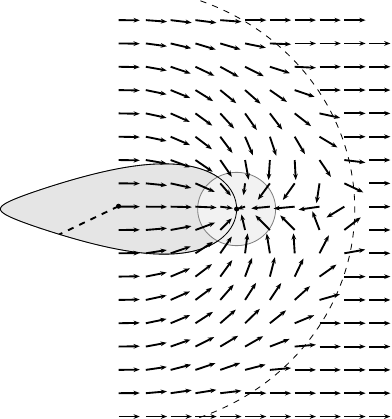}
			\put(24,52){$p$}     
			\put(-4,32){\small $D_X(p,r(p,\phi))$}
			\put(65,56.2){\small $ T_{\varepsilon}$}
			\put(5,48){\small $r(p,\phi)$}
		\end{overpic}
		\caption{Definition of the photography map as a tiny
			vortex ring in a tubular neighborhood $T_\varepsilon$
		of $\partial D_X(p,r(p,\phi)).$}
		\label{fig:photo2}
	\end{subfigure}
	\caption{Construction of the photography map:
		using a cylindrical coordinate system in a local chart,
		we define $f_{\varepsilon,\phi}(p)$ by transporting
		the dipole vortex map $\omega_{p,\phi}$ onto $M$.
		The singularity, located at a distance $r(p,\phi)$,
		is shifted to the boundary $\partial D_X(p,r(p,\phi))$
		and then removed by a linear smoothing.
	}
	\label{fig:recovery-seq}
\end{figure}
\begin{proof}
	We rely on the standard vortex ring construction of
	the recovery sequence for the Ginzburg–Landau functional
	(cf. \cite{BOS2004}; see also \cite[Prop.~16]{dephilippis-pigati}).
	The two differences with respect to the cited constructions are that we work on a Riemannian manifold and that we have to take into account the 
	flux constraint.

	Since the photography map will be constant outside $B_g(p,2r(p,\phi))$ and 
	$2\,r(p,\phi)<2r_0<\operatorname{inj}(M)$,
	we can work in a single normal chart via the exponential map.
	To this end, choose a linear isometry $F:\mathbb{R}^N\to T_pM$
	such that $F(e_1) = X(p)$, and use 
	spherical coordinates $(r,\theta) \in [0,\infty)\times\mathbb S^{N-2}$
	for the remaining variables.
	Then, setting  $U\coloneqq \big\{(x,r,\theta) \in \mathbb{R} \times [0,+\infty) \times \mathbb{S}^{N-2}:\; x^2+r^2< 4\, r_0^2\big\}$,
	we define the map $\varphi\colon U \to M $ as 
	\[
		\varphi(x,r,\theta) = \exp_p\big(F(x e_1 + r\, \theta)\big).
	\]
	By construction, this map is a diffeomorphism between $U$
	and $B_g(p,2r_0)$.
	In this local chart, $\partial D_X(p,r(p,\phi))$
	corresponds to
	$\{(0,r(p,\phi),\theta): \theta \in \mathbb{S}^{N-2} \} \subset U$.
	In what follows, we implicitly identify points of $B_g(p,2r_0)$
	with their coordinates $(x,r,\theta)$.

	We are going to follow the same lines of~\cite{BOS2004},
	but we have to consider also the deformations given by
	the curvature of the underlying manifold.
	As the photography map $f_{\varepsilon,\phi}(p)$
	will be independent on $\theta$, 
	we identify $(x,r)$ with the complex variable $z=x+i\,r$ 
	and define the dipole vortex map 
	$\omega_{p,\phi}:\mathbb{C}\to\mathbb{C}$ by
	\begin{equation}
		\label{eq:def-w_R}
		\omega_{p,\phi}(z) \coloneqq
		\frac{z- i\, r(p,\phi)}{|z - i\, r(p,\phi)|}
		\frac{z +  i\, r(p,\phi)}{|z + i\, r(p,\phi)|}
		e^{\displaystyle i\vartheta(z)},
	\end{equation}
	where $\vartheta_p: \mathbb{C}\to \mathbb{R}$
	is a real harmonic function such that 
	$\omega_{p,\phi} \equiv 1$ on $\partial B_{2r(p,\phi)}$.
	A graphic representation of this function is given in Figure~\ref{fig:photo1}.
	In the half–plane $r\ge0$,
	$\omega_{p,\phi}$ has a single singularity at $z=i\,r(p,\phi)$, 
	which corresponds to the points of $\partial_X D(p,r(p,\phi))$.
	We remove this singularity by using a linear smoothing 
	in a small disk of radius $\varepsilon$;
	hence, let us set
	\begin{multline*}
		T_{\varepsilon,\phi}(p) \coloneqq 
		\big\{(x,\rho,\theta) \in U:
		x^2 + (r - r(p,\phi))^2 \le \varepsilon^2 \big\}\\
		=
		\big\{q \in M: \operatorname{dist}_g(q,\partial D_X(p,r(p,\phi))\le \varepsilon\big\}.
	\end{multline*}
	Therefore, for any $\varepsilon < r(p,\phi)/2$,
	we define the map $\hat{f}_{\varepsilon,\phi}(p)\colon M \to \mathbb{C}$
	as follows:
	\begin{equation}
		(\hat{f}_{\varepsilon,\phi}(p))(q)
		= \begin{cases}
			1, & \mbox{if }\operatorname{dist}_g(q,p) \ge 2\, r(p,\phi),\\
			\varepsilon^{-1}
			\big|x + i \big(r -r(p,\phi)\big) \big|\,
			\omega_r(x + i \rho), & \mbox{if }
			q \in T_{\varepsilon,\phi}(p),\\
			\omega_r(x + i r), & \mbox{otherwise},
		\end{cases}
	\end{equation}
	where $q=\varphi(x,\rho,\theta)$ in the last two cases. 
	A graphic representation of this map is given in Figure~\ref{fig:photo2}
	(for a fixed $\theta$).
	By compactness of $\Sigma$ and continuity of $p\mapsto r(p,\phi)>0$,
	set
	\[
		\varepsilon_0 = \varepsilon_0(\phi)\coloneqq \frac{1}{2} \min_{p\in\Sigma} r(p,\phi) \;>\; 0,
	\]
	so that $\hat{f}_{\varepsilon,\phi}(p)$ is well defined for every $p\in\Sigma$ and
	$\varepsilon\in(0,\varepsilon_0(\phi))$.
	Since the whole above construction relies on continuous maps,
	$p \mapsto \hat{f}_{\varepsilon,\phi}(p)$ is a continuous map
	from $\Sigma$ to $W^{1,2}(M,\mathbb{C})$.

	Let us show~\eqref{eq:photo-almost-flux}.
	As $|\hat{f}_{\varepsilon,\phi}| \equiv 1$
	outside $T_{\varepsilon,\phi}(p)$,
	we have $J\big(\hat f_{\varepsilon,\phi}(p)\big)=0$
	there, and so we obtain
	\begin{equation}
		\label{eq:photo-Phi-hatf-estimate}
		\Phi_X\big(\hat f_{\varepsilon,\phi}(p)\big)
		= \frac{1}{\pi}
		\int_{\textstyle T_{\varepsilon,\phi}(p)} \langle J(\hat f_{\varepsilon,\phi}(p)),\, \star Y \rangle \, \mathrm{d}v_g.
	\end{equation}
	Since in $T_{\varepsilon,\phi}(p)$
	the function $\hat{f}_{\varepsilon,\phi}(p)$
	can be approximated by a vortex of degree 1
	(see also Figure~\ref{fig:photo2}),
	namely
	\[
		\hat{f}_{\varepsilon,\phi}(p)
		(x,r,\phi)
		\approx \frac{1}{\varepsilon} 
		\frac{ix - (r - r(p,\phi))}{|x - i(r - r(p,\phi))|}
		e^{ i\vartheta(r(p,\phi))},
	\]
	we obtain 
	\[
		J(\hat{f}_{\varepsilon,\phi}(p)) = \frac{1}{\varepsilon^2} \,\mathrm{d}x \wedge \mathrm{d}r
		+ O(1/\varepsilon),
		\qquad \mbox{in }T_{\varepsilon,\phi}(p).
	\]
	Hence, for any smooth test 2-form $Z \in \Omega^{2}(M)$
	we obtain
	\begin{multline*}
		\lim_{\varepsilon \to 0}
		\int_{\textstyle T_{\varepsilon,\phi}(p)}
		\langle J(\hat f_{\varepsilon,\phi}(p)),\, Z \rangle \, \mathrm{d}v_g
		=
		\lim_{\varepsilon \to 0}
		\frac{1}{\varepsilon^2}
		\int_{\textstyle T_{\varepsilon,\phi}(p)}
		\langle \mathrm{d}x \wedge \mathrm{d}r ,\,  Z \rangle \,
		\mathrm{d}v_g  \\
		= 
		\lim_{\varepsilon \to 0}
		\frac{\operatorname{vol}(T_{\varepsilon,\phi}(p))}{\varepsilon^2}
		\llbracket \partial_X D(p,r(p,\phi))\rrbracket
		(Z) 
		= 
		\pi \llbracket \partial_X D(p,r(p,\phi))\rrbracket
		(Z);
	\end{multline*}
	in other words
	\[
		\star J(\hat{f}_{\varepsilon,\phi}(p))
		\rightharpoonup \pi \llbracket \partial_X D(p,r(p,\phi))\rrbracket,
	\]
	in the sense of currents.  
	Therefore, by Stokes' Theorem
	and recalling the definition of $r(p,\phi)$ in Lemma~\ref{lem:rightOrthoDisk}, we have
	\[
		\lim_{\varepsilon \to 0} \Phi_X\big(\hat{f}_{\varepsilon,\phi}(p)\big)
		= \llbracket \partial D_X\big(p,r(p,\phi)\big)\rrbracket(Y) \\
		= \llbracket D_X(p,r(p,\phi)) \rrbracket (\star X)
		= \phi,
	\]
	so~\eqref{eq:photo-almost-flux} holds.

	It remains to prove~\eqref{eq:limEnergyPhoto-hat}.
	Since $\|\hat{f}_{\varepsilon,r}(p)\| = 1$
	outside $T_{\varepsilon,\phi}(p)$
	there we have $W(\hat{f}_{\varepsilon,r}(p)) \equiv 0$.
	Moreover, $\|\hat{f}_{\varepsilon,r}(p)\| \le 1$
	inside $T_{\varepsilon,\phi}(p)$
	and so $W(\hat{f}_{\varepsilon,r}(p))$
	is upper-bounded by a constant.
	As a consequence,
	\[
		\int_M W(\hat f_{\varepsilon,r})\,\mathrm{d}v_g
		= O\big(\mathcal{H}^N(T_{\varepsilon,\phi}(p))\big)
		= O(\varepsilon^2),
	\]
	and so we obtain
	\[
		\lim_{\varepsilon\to 0}
		\frac{1}{\pi\varepsilon^2|\log\varepsilon|}
		\int_M W(\hat{f}_{\varepsilon,r})\mathrm{d}v_g
		= \lim_{\varepsilon\to 0}\frac{1}{\pi|\log\varepsilon|} O(1)
		= 0,
	\]
	which means that potential term does not contribute to the limit in~\eqref{eq:limEnergyPhoto-hat}.
	For the gradient term, we notice that 
	$\| \nabla \hat{f}_{\varepsilon,r}(p)\| = \varepsilon^{-1}$
	in $T_{\varepsilon,\phi}(p)$,
	while it is less than a constant outside of it.
	Therefore, setting $\varepsilon_0 = \varepsilon_0(\phi)$,
	we have
	\[
		\lim_{\varepsilon \to 0}
		\frac{1}{2\pi|\log\varepsilon|}
		\int_{\textstyle M \setminus T_{\varepsilon_0,\phi}(p)}
		\|\nabla\hat{f}_{\varepsilon,r}(p)\|^2\, \mathrm{d}v_g = 0,
	\]
	and so it remains to estimate the following quantity:
	\[
		\lim_{\varepsilon\to 0}
		\frac{1}{2\pi|\log\varepsilon|}
		\int_{\textstyle T_{\varepsilon_0,\phi}(p)}
		\|\nabla\hat{f}_{\varepsilon,\phi}(p)\|^2\, \mathrm{d}v_g.
	\]
	Decomposing $\nabla\hat f_{\varepsilon,r}(p)=\nabla(|\hat f_{\varepsilon,r}(p)|)\,\omega_{p,\phi}
	+ |\hat f_{\varepsilon,r}(p)|\,\nabla\omega_{p,\phi}$, the term with $\nabla(|\hat f_{\varepsilon,r}(p)|)$
	is supported in the $T_{\varepsilon,\phi}(p)$,
	where its norm is $\varepsilon^{-1}$ and so
	and we have
	\[
		\int_{\textstyle T_{\varepsilon,\phi}(p)} \|\nabla(|\hat f_{\varepsilon,r}
		(p)|)\|^2 \, \mathrm{d}v_g
		= \frac{1}{\varepsilon^2} \mathcal{H}^{N-2}(T_{\varepsilon,\phi}(p))
		= O(1),
		\quad\mbox{as }\varepsilon \to 0,
	\]
	and it goes to zero after division by $|\log\varepsilon|$.
	The dominant contribution comes from $|\nabla\omega_{p,\phi}|^2$ on the annulus
	between radii $\varepsilon$ and $\varepsilon_0$ in the normal $(x,r)$–plane.
	By the standard 2D vortex estimate,
	\[
		\frac{1}{2} \int_{\textstyle   
		B(i \,r(p,\phi),\varepsilon_0) \setminus B(i \,r(p,\phi), \varepsilon) }
		\,|\nabla\omega_{p,\phi}|^2 \,\mathrm{d}x\, \mathrm{d}r 
		=\pi\,|\log\varepsilon| + O(1).
	\]
	Integrating over the tangential directions $\theta\in \mathbb{S}^{N-2}$ and using the normal–coordinate
	approximation of $\mathrm{d}v_g$, we obtain
	\[
		\lim_{\varepsilon \to 0}
		\frac{1}{2\pi|\log\varepsilon|}
		\int_{\textstyle T_{\varepsilon_0,\phi}(p)}
		\|\nabla\hat f_{\varepsilon,r}(p)\|^2\,\mathrm{d}v_g
		= \mathcal{H}^{N-2}\big(\partial D_X(p,r)\big),
	\]
	which yields \eqref{eq:limEnergyPhoto-hat}.
\end{proof}

\begin{lemma}
	\label{lem:photoConstruction}
	For any $\phi \in (0,\phi_0)$
	and $\varepsilon_0 = \varepsilon_0(\phi) $
	given by Lemma~\ref{lem:photoConstruction-hat},
	there exists and a one-parameter 
	family of continuous functions 
	\[		
		(f_{\varepsilon,\phi})_{\varepsilon \in (0,\varepsilon_0)} \colon \Sigma \to \mathcal{X}_{\phi}
	\]
	such that
	\begin{equation}
		\label{eq:limEnergyPhoto}
		\lim_{\varepsilon \to 0} E_{\varepsilon,\phi}(f_{\varepsilon,\phi}(p))
		= \mathcal{H}_g^{N-2}\big(\partial D_X(p,r(p,\phi))\big).
	\end{equation}
\end{lemma}

\begin{proof}
	We will define 
	$f_{\varepsilon,\phi}\colon \Sigma \to \mathcal{X}_\phi \subset W^{1,2}(M,\mathbb{C})$
	as a small phase perturbation of 
	the map $\hat{f}_{\varepsilon,\phi}$
	given by Lemma~\ref{lem:photoConstruction-hat}.
	Indeed, as the flux error goes to zero,
	namely
	\[
		\lim_{\varepsilon \to 0}
		\Phi_X\big(\hat{f}_{\varepsilon,\phi}(p)\big)
		- \phi = 0,
	\]
	it suffices to add a small phase perturbation to fix it:
	in other words, we will define 
	$f_{\varepsilon,\phi}\colon \Sigma \to \mathcal{X}_\phi \subset W^{1,2}(M,\mathbb{C})$
	as 
	\begin{equation}
		\label{eq:photo-phase-perturbation}
		f_{\varepsilon,\phi}(p) = e^{\textstyle i \tau_{p,\varepsilon}}\, \hat{f}_{\varepsilon,\phi}(p),
	\end{equation}
	where the function $\tau_{p,\varepsilon} \colon M \to \mathbb{R}$
	will be such that 
	$\Phi_X\big(f_{\varepsilon,\phi}(p)\big) = \phi$.

	By a standard computation, we have
	\[
		j\big(f_{\varepsilon,\phi}(p)\big) = 
		j\big(\hat{f}_{\varepsilon,\phi}(p)\big) + 
		\norm{\hat{f}_{\varepsilon,\phi}(p)}^2 \nabla \tau_{p,\varepsilon},
	\]
	so by setting $\Phi_X(f_{\varepsilon,\phi}(p)) = \phi$
	we obtain the condition
	\begin{equation}
		\label{eq:perturbation-condition}
		\int_M \norm{\hat{f}_{\varepsilon,\phi}(p)}^2
		g(\nabla \tau_{p,\varepsilon}, X)
		\mathrm{d}v_g
		= \phi - \Phi_X\big(\hat{f}_{\varepsilon,\phi}(p)\big),
	\end{equation}
	where the right-hand side is infinitesimal
	as $\varepsilon \to 0$ by~\eqref{eq:photo-almost-flux}.
	Let us choose a continuous function
	$\tau_p \colon M \to \mathbb{R}$
	with compact support in a small ball centered at $p$
	and such that 
	\[
		\int_M g(\nabla\tau_p,X) \mathrm{d}v_g \ne 0.
	\]
	For instance, $\tau_p$
	can be chosen as the solution of the following 
	Dirichlet problem
	\[
		\begin{dcases}
			\Delta \tau_p = 
			- \operatorname{div}((\varepsilon_0-\operatorname{dist}_g(q,p))X(q)),&
			\mbox{in } B_g(p,\varepsilon_0),\\
			\tau_p = 0, &
			\mbox{in } \partial B_g(p,\varepsilon_0).\\
		\end{dcases}
	\]
	In this way, since $\norm{\hat{f}_{\varepsilon,\phi}(p)} \equiv 1$ in this small ball, we have that 
	\[
		c_{p,\varepsilon}
		= \frac{\phi - \Phi_X\big(\hat{f}_{\varepsilon,\phi}(p)\big)}{
		\int_M \norm{\hat{f}_{\varepsilon,\phi}(p)}^2  g(\nabla\tau_p,X) \mathrm{d}v_g} \in \mathbb{R}
	\]
	is well defined and the function 
	$\tau_{p,\varepsilon} = c_{p,\varepsilon}\tau_p$
	satisfies~\eqref{eq:perturbation-condition}.
	Moreover, by~\eqref{eq:photo-almost-flux} we have that
	$c_{p,\varepsilon} \to 0$ as $\varepsilon \to 0$,
	and so 
	\begin{equation}
		\label{eq:photo-estimate-nablatheta}
		\lim_{\varepsilon \to 0} \|\nabla \tau_{p,\varepsilon} \|_{L^\infty(M)} = 0.
	\end{equation}

	To obtain~\eqref{eq:limEnergyPhoto},
	it suffices to prove that the phase perturbation applied 
	to fulfill the flux constraint 
	does not affect the limit of the functional.
	As it is only a phase perturbation,
	the norm of $f_{\varepsilon,\phi}(p)$
	is always equal to the one of
	$\hat{f}_{\varepsilon,\phi}(p)$,
	and so the potential term vanishes in the limit,
	namely
	\begin{equation}
		\label{eq:dependence-eps-W}
		\lim_{\varepsilon \to 0}
		\frac{1}{\pi|\log\varepsilon|}
		\int_M \frac{1}{\varepsilon^2}W(f_{\varepsilon,\phi(p)})
		\,\mathrm{d}v_g 
		=
		\lim_{\varepsilon \to 0}
		\frac{1}{\pi|\log\varepsilon|}
		\int_M \frac{1}{\varepsilon^2}W(\hat{f}_{\varepsilon,\phi(p)})
		\,\mathrm{d}v_g 
		= 0.
	\end{equation}
	For the gradient part, we have
	\[
		\| \nabla f_{\varepsilon,\phi}(p) \|^2
		= \| \nabla \hat{f}_{\varepsilon,\phi}(p) \|^2
		+ \| \hat{f}_{\varepsilon,\phi}(p) \|^2
		\|\nabla \tau_{p,\varepsilon}\|^2 
		+ 2 g\big(j(\hat{f}_{\varepsilon,\phi}(p)),\nabla \tau_{p,\varepsilon}\big).
	\]
	Therefore, by using~\eqref{eq:photo-estimate-nablatheta}
	and recalling that $\| \hat{f}_{\varepsilon,\phi}(p) \| \le 1$,
	we have
	\[
		\lim_{\varepsilon \to 0}
		\frac{1}{2\,\pi|\log\varepsilon|}
		\int_M \| \hat{f}_{\varepsilon,\phi}(p) \|^2
		\|\nabla \tau_{p,\varepsilon}\|^2 
		\mathrm{d}v_g = 0,
	\]
	and, recalling also the estimate 
	for $\int_M \|\nabla \hat{f}_{\varepsilon,\phi}(p)\|^2 \mathrm{d}v_g$ as $\varepsilon \to 0$, 
	we obtain
	\begin{multline*}
		\lim_{\varepsilon \to 0}
		\bigg|
		\frac{1}{\,\pi|\log\varepsilon|}
		\int_M 
		g\big(j(\hat{f}_{\varepsilon,\phi}(p)),\nabla \tau_{p,\varepsilon}\big)
		\,\mathrm{d}v_g 
		\bigg| \\
		\le  \lim_{\varepsilon \to 0}
		\frac{1}{\,\pi|\log\varepsilon|}
		\int_M 
		\|\hat{f}_{\varepsilon,\phi}(p) \|
		\|\nabla \hat{f}_{\varepsilon,\phi}(p)\|
		\|\nabla \tau_{p,\varepsilon}\|\,
		\mathrm{d}v_g \\
		\le \lim_{\varepsilon \to 0}
		\frac{1}{\,\pi|\log\varepsilon|}
		\bigg(
			\int_M 
			\|\nabla \hat{f}_{\varepsilon,\phi}(p)\|^2
			\,\mathrm{d}v_g 
		\bigg)^{1/2}
		\bigg(
			\int_M 
			\|\nabla \tau_{p,\varepsilon}\|^2
			\,\mathrm{d}v_g 
		\bigg)^{1/2} \\
		=
		\lim_{\varepsilon \to 0}
		\frac{\Big(2\, \mathcal{H}^{N-2}(\partial D_X(p,r(p,\phi))\Big)^{1/2}}{\sqrt{\pi|\log\varepsilon|}}
		\bigg(
			\int_M 
			\|\nabla \tau_{p,\varepsilon}\|^2
			\,\mathrm{d}v_g 
		\bigg)^{1/2}
		= 0.
	\end{multline*}
	Using the last estimate, we then get
	\begin{multline*}
		\lim_{\varepsilon \to 0}
		\frac{1}{2\pi|\log\varepsilon|}
		\int_{M}
		\|\nabla f_{\varepsilon,r}(p)\|^2\,\mathrm{d}v_g \\
		= 
		\lim_{\varepsilon \to 0}
		\frac{1}{2\pi|\log\varepsilon|}
		\int_{M}
		\|\nabla\hat f_{\varepsilon,r}(p)\|^2\,\mathrm{d}v_g
		= \mathcal{H}^{N-2}\big(\partial D_X(p,r)\big),
	\end{multline*}
	and~\eqref{eq:limEnergyPhoto} holds.
\end{proof}

As a consequence of Lemma~\ref{lem:photoConstruction},
to estimate a small sublevel of the functional 
$E_{\varepsilon,\phi}$ that contains the image of the photography map $f_{\varepsilon,\phi}$,
we need an estimate of the measure of
$\partial D_X(p,r(p,\phi))$.
This estimate can be obtained by combining 
the asymptotic expansion for small volumes of the surface area of the geodesic balls together with 
an asymptotic expansion of the volume of 
$D_X(p,r(p,\phi))$ with respect to the flux parameter $\phi$.
\begin{lemma}
	\label{lem:estimate-partialDphi}
	There exists a continuous function 
	$\Upsilon \colon \Sigma \to \mathbb{R}$ 
	such that, for any $p \in \Sigma$, we have
	\begin{equation}
		\label{eq:estimate-partialD_X}
		\mathcal{H}^{N-2}_g \big(\partial D_X(p,r(p,\phi))\big)
		= \gamma_{N-1}\phi^{\frac{N-2}{N-1}}
		+ \Upsilon(p)\phi + \omega(\phi),
	\end{equation}
	where the remainder term $\omega(\phi)$
	is independent of $p \in \Sigma$
	and satisfies $\omega(\phi)/\phi \to 0$
	as $\phi \to 0$. 
\end{lemma}
\begin{proof}
	For $r \in (0,r_0)$,
	consider the disks $D_X(p,r)$
	as geodesics balls of an
	$(N-1)$--dimensional submanifold of $M$.
	We denote by 
	$v_p(\phi)$ the $(N-1)$--dimensional volume
	of $D_X(p,r(p,\phi))$, hence
	\[
		v_p(\phi) = 
		\mathcal{H}^{N-1}_g\big(D_X(p,r(p,\phi))\big)
		=
		\int_{D_X(p,r(p,\phi))} 1 \, \mathrm{d}\sigma,
	\]
	where $\mathrm{d}\sigma$ is the induced volume form.
	By the asymptotic expansion of surface area
	as a function of the enclosed volume 
	for small geodesic balls
	(see, e.g.,~\cite[Lemma 3.10]{Nardulli2009-AGAG}),
	we obtain 
	\[
		\mathcal{H}^{N-2}_g \big(\partial D_X(p,r(p,\phi))\big)
		= \gamma_{N-1}v_p(\phi)^{\frac{N-2}{N-1}}
		+ a_p v_p(\phi)^{\frac{N}{N-1}} + o\big(v_p(\phi)^{\frac{N}{N-1}}\big),
	\]
	where $a_p \in \mathbb{R}$ depends on the scalar curvature of the submanifold.
	Therefore, to prove~\eqref{eq:estimate-partialD_X}
	it suffices to show the existence of a 
	continuous functions $h\colon \Sigma \to \mathbb{R}$
	such that
	\begin{equation}
		\label{eq:expansion-vphi}
		v_p(\phi) = \phi + h(p) \phi^{\frac{N}{N-1}} +
		o(\phi^{\frac{N}{N-1}}).        
	\end{equation}
	Indeed, combining the last two asymptotic expansions
	and using the Taylor series of power
	we get
	\begin{multline*}
		\mathcal{H}^{N-2}_g \big(\partial D_X(p,r(p,\phi))\big)
		= \gamma_{N-1}
		\phi^{\frac{N-2}{N-1}}
		\big(1 + h(p)\phi^{\frac{1}{N-1}}\big)^{\frac{N-2}{N-1}}
		+ o(\phi) \\
		=
		\gamma_{N-1}\phi^{\frac{N-2}{N-1}}
		+ \gamma_{N-1} \frac{N-2}{N-1}\, h(p)\phi + o(\phi),
	\end{multline*}
	from which~\eqref{eq:estimate-partialD_X} follows.
	To prove~\eqref{eq:expansion-vphi},
	let us define $\eta_p \colon D_X(p,r_0) \to \mathbb{R}$
	as follows:
	\[
		\eta_p(q) := 1 - g\big(X(q),\nu_p(q)\big),
	\]
	which is a map of class $C^1$ vanishing at $p$.
	By definition and by using the coarea formula we get
	\[
		v_p(\phi) - \phi 
		= \int_{D_X(p,r(p,\phi))} \eta_p(q) \,\mathrm{d}\sigma
		= \int_0^{r(p,\phi)} 
		\Big(\int_{\partial{D(p,s)}} 
		\eta_p(q) \,\mathrm{d}\mathcal{H}^{N-2}(q)\Big)
		\,
		\mathrm{d}s.
	\]
	Let us notice that the map
	\[
		s \mapsto \int_{\textstyle \partial{D(p,s)}} 
		\eta_p(q) \,\mathrm{d}\mathcal{H}^{N-2}(q)
	\]
	is continuous by the continuity of $\eta_p$
	and of the measure of the geodesic spheres.
	Moreover, as $\eta_p$ is of class $C^1$,
	we have $\eta_p(q) \sim s$ for any $q \in \partial{D(p,s)}$, 
	so that, by using also the intermediate value theorem, 
	there exists a continuous real function $h_p$
	such that
	\[
		\int_{\partial{D(p,s)}} 
		\eta_p(q) \,\mathrm{d}\mathcal{H}^{N-2}(q)
		= h_p(s)s^{N-1}. 
	\]
	As a consequence, 
	using again the intermediate value theorem,
	there exists a continuous function $h(p,\phi)$
	such that
	\[
		v_p(\phi) - \phi
		= \int_0^{r(p,\phi)}
		h_p(s)s^{N-1}\, \mathrm{d}s
		= h(p,\phi) \, r(p,\phi)^N.
	\]
	Hence, 
	recalling that $\phi$ behaves like $r(p,\phi)^{N-1}$ as $\phi \to 0$
	(see~\eqref{eq:rpphi-estimate}),
	we conclude that
	\[
		\lim_{\phi \to 0} \frac{v_p(\phi) - \phi}{\phi^{N/(N-1)}} 
		= \lim_{\phi \to 0}
		h(p,\phi)
		\frac{r(p,\phi)^{N}}{\phi^{N/N-1}}
		\eqqcolon h(p),
	\]
	obtaining~\eqref{eq:expansion-vphi}.
	Finally, since all the quantities involved in the expansion depend smoothly on $p \in \Sigma$, and $\Sigma$ is compact, the remainder term $\omega(\phi)$ in~\eqref{eq:estimate-partialD_X} can be taken uniformly in $p$.
\end{proof}

By combining the above lemmas, 
we are now ready to state and proof the main result
of this section,
thus providing an estimate of a small
sublevel that contains the entire image of the photography map.
\begin{proposition}
	\label{prop:photography-sublevel}
	Let $f_{\varepsilon,\phi}\colon \Sigma \to \mathcal{X}_{\phi}$
	be defined as in
	Lemma~\ref{lem:photoConstruction}.
	There exists $\alpha = \alpha(M,g,X) > 0$ and $\phi_1 \in (0,\phi_0)$
	such that, for every $\phi \in (0,\phi_1)$ there exists 
	$\varepsilon_1 = \varepsilon_1(\phi) \in (0,\varepsilon_0(\phi))$ such that for every
	$\varepsilon \in (0,\varepsilon_1)$ we have
	\[
		E_{\varepsilon,\phi}(f_{\varepsilon,\phi}(p))
		\le \gamma_{N-1}\phi^{\frac{N-2}{N-1}} + \alpha \phi,
		\qquad \forall p \in \Sigma.
	\]
	In other words, setting
	$c\colon (0,\phi_1) \to \mathbb{R}$ as
	\begin{equation}
		\label{eq:def-sublevel}
		c(\phi)\coloneqq \gamma_{N-1}\phi^{\frac{N-2}{N-1}} + \alpha \phi,
	\end{equation}
	we have
	\[
		f_{\varepsilon,\phi}(\Sigma)\subset E_{\varepsilon,\phi}^{c(\phi)},
		\qquad \forall \phi \in (0,\phi_1),
		\, \forall\varepsilon \in (0,\varepsilon_1(\phi)).
	\]
\end{proposition}
\begin{proof}
	By Lemma~\ref{lem:estimate-partialDphi},
	since $\Upsilon\colon \Sigma \to \mathbb{R}$
	is a continuous function and 
	$\Sigma$ is compact, it attains its maximum.
	Moreover, 
	there exists $\phi_1 \in (0,\phi_0)$
	such that 
	the remainder term in~\eqref{eq:estimate-partialD_X} satisfies $\omega(\phi) \le \phi/2$
	for any $\phi \in (0,\phi_1)$.
	Therefore, for every $p \in \Sigma$ and $\phi \in (0,\phi_1)$,
	estimate~\eqref{eq:estimate-partialD_X} gives
	\[
		\mathcal{H}^{N-2}_g \big(\partial D_X(p,r(p,\phi))\big)
		\le \gamma_{N-1}\phi^{\frac{N-2}{N-1}}
		+ \big(\max_{p \in \Sigma}\Upsilon(p)\big) \phi
		+ \frac{\phi}{2}.
	\]
	By combining this last estimate with~\eqref{eq:limEnergyPhoto},
	we obtain 
	\[
		E_{\varepsilon,\phi}\big(f_{\varepsilon,\phi}(p)\big)
		\le \gamma_{N-1}\phi^{\frac{N-2}{N-1}}
		+ \big(\max_{p \in \Sigma}\Upsilon(p)\big)\phi 
		+ \frac{\phi}{2}
		+ o(\varepsilon),
	\]
	as $\varepsilon \to 0$.
	Setting $\varepsilon_1(\phi) \in (0,\varepsilon_0)$
	such that $o(\varepsilon) \le \phi/2$ for any
	$\varepsilon \in (0,\varepsilon_1(\phi))$ we obtain
	for every $p \in \Sigma$ the following estimate
	\[
		E_{\varepsilon,\phi}\big(f_{\varepsilon,\phi}(p)\big)
		\le \gamma_{N-1}\phi^{\frac{N-2}{N-1}}
		+ \big(\max_{p \in \Sigma}\Upsilon(p)\big)\phi 
		+ \phi,
		\qquad \forall \phi \in (0,\phi_1),\,
		\forall \varepsilon \in (0,\varepsilon_1(\phi)),
	\]
	and so we get the thesis by setting
	$\alpha = \max_{p \in \Sigma}\Upsilon(p) +1$.
\end{proof}

\begin{remark}\label{rem:dependence_eps}
	A close inspection of the proof of
	Proposition~\ref{prop:photography-sublevel} shows that the parameter
	$\varepsilon_1$ depends on on both $\phi$ and the potential~$W$.
	Indeed, $\varepsilon_1(\phi)$ is chosen so that the energy
	$E_{\varepsilon,\phi}(f_{\varepsilon,\phi}(p))$ stays within $\phi/2$
	of its limiting value
	$\gamma_{N-1}\phi^{\frac{N-2}{N-1}}
	+\big(\max_{p\in\Sigma}\Upsilon(p)\big)\phi$,
	a choice that relies on~\eqref{eq:limEnergyPhoto}.
	In that estimate, the contribution of the potential term
	vanishes in view of~\eqref{eq:dependence-eps-W},
	so that the resulting error term depends explicitly on~$W$.
	This explains the dependence on~$W$
	of the parameter $\varepsilon^*$
	appearing in Theorem~\ref{theorem:main}.
	On the other hand, since the photography map
	must be well defined, we set $\varepsilon_1(\phi) \le \varepsilon_0$,
	where $\varepsilon_0(\phi)$ depends also on $M$, $g$, and~$X$,
	as can be checked from its definition in the proof of 
	Lemma~\ref{lem:photoConstruction}.
\end{remark}

\begin{remark}
	\label{rem:firstorder-cphi}
	By combining~\eqref{eq:J_M-asymptotic}
	and~\eqref{eq:def-sublevel},
	we obtain 
	\begin{equation}
		\label{eq:firstorder-cphi}
		\lim_{\phi \to 0}
		\frac{c(\phi)}{J_M(\phi,X)} =
		\lim_{\phi \to 0}
		\frac{\gamma_{N-1}\phi^{\frac{N-2}{N-1}} + \alpha \phi}{\gamma_{N-1}\phi^{\frac{N-2}{N-1}}
		+  o(\phi^{\frac{N-2}{N-1}})} = 1,
	\end{equation}
	which will be a crucial estimate for the next section.
\end{remark}

\section{Barycenter map}
\label{sec:barycenter}

In this section, we prove the 
existence of a well-defined and continuous barycenter map on the sublevel set containing the image of the photography map and taking values in~$\Sigma$,
provided that $\phi$ and $\varepsilon$ are sufficiently small.
This leads to the following result, where we recall that the function
$c\colon (0,\phi_1) \to \mathbb{R}$
has been defined in~\eqref{eq:def-sublevel}.

\begin{proposition}
	\label{prop:barycenter}
	There exists $\phi_2 \in (0,\phi_1)$
	such that for any $\phi \in (0,\phi_2)$
	there exists $\varepsilon_2 = \varepsilon_2(\phi) \in (0,\varepsilon_1)$ such that
	for any $\varepsilon \in (0,\varepsilon_2)$
	there exists a continuous map 
	\[
		\beta\colon  E_{\varepsilon,\phi}^{c(\phi)}\to \Sigma.
	\]
\end{proposition}

To prove Proposition~\ref{prop:barycenter}, we first establish that the \emph{energy density} of any function in the sublevel set is mostly concentrated within a small ball
centered at a point near $\Sigma$,
where the energy density $e_\varepsilon: \mathcal{X}_\phi \to L^1(M)$ is defined
for any $\varepsilon > 0$
as
\begin{equation}
	\label{eq:def-energydensity}
	e_\varepsilon(u)(x) = 
	\frac{1}{\pi |\log\varepsilon|}
	\left( \frac{1}{2}\norm{\nabla u}^2
		+ \frac{1}{\varepsilon^2}W(u)
	\right),
\end{equation}
so that
\[
	E_{\varepsilon,\phi}(u) = 
	\int_M e_\varepsilon(u) \, \mathrm{d}v_g,
	\qquad \forall u \in \mathcal{X}_\phi.
\]
This concentration property is stated in the following result.
\begin{lemma}
	\label{lem:jacobian-concentration}
	There exists a constant $\mu > 0$
	such that the following property holds.  
	For any $\eta \in (1/2,1)$
	there exists 
	$\phi_3 \in (0,\phi_1)$ such that
	for any $\phi \in (0,\phi_3)$
	there exists $\varepsilon_3 = \varepsilon_3(\phi) \in (0,\varepsilon_1)$
	such that for any $\varepsilon \in (0,\varepsilon_3)$
	and any function $u \in E_{\varepsilon,\phi}^{c(\phi)}$
	there exists a point $p \in \Sigma_{\delta/2}$
	(where $\delta > 0$ 
	is given by~\ref{ass:Sigmadelta-retraction})
	such that
	\begin{equation}
		\label{eq:jacobian-concentration}
		\frac{
			\int_{B_g(p,\mu \,\phi^{\frac{1}{N-1}})}
			e_\varepsilon(u) \mathrm{d}v_g
			}{
			\int_{M}
			e_\varepsilon(u) \mathrm{d}v_g
		}
		\ge \eta.
	\end{equation}

\end{lemma}
The proof of Lemma~\ref{lem:jacobian-concentration}
relies on the equicoerciveness and
liminf properties of $\Gamma$-convergence 
to invoke a similar concentration result
for the almost minimizers of the isoperimetric problem 
in codimension 2, which is
Theorem~\ref{cor:almost-iso-concentration}.
To maintain the flow of the discussion, and since the result is interesting in its own right, we postpone the statement and the proof of 
Theorem~\ref{cor:almost-iso-concentration} 
to Appendix~\ref{app:generalized-compactness}.

\begin{proof}[Proof of Lemma~\ref{lem:jacobian-concentration}]
	By contradiction, assume that there exists $\eta \in (1/2,1)$ as well as sequences
	$(\phi_n)_{n \in \mathbb{N}}$ in $(0,1)$, $(\varepsilon_{n,k})_{(n,k) \in \mathbb{N}^2}$ in $(0,1)$ and 
	$(u_{n,k})_{(n,k) \in \mathbb{N}^2}$ in $W^{1,2}(M,\mathbb{C})$
	such that
	\begin{equation*}
		\lim_{n \to \infty}\phi_n=0, \quad \lim_{k \to \infty}\varepsilon_{n,k}=0 \mbox{ for all } n \in \mathbb{N}, \quad u_{n,k} \in E_{\varepsilon_{n,k},\phi_n}^{c(\phi_n)},
	\end{equation*}
	and
	\begin{equation*}
		\frac{
			\int_{B_g(p,\mu\phi_n^{\frac{1}{N-1}})}
			e_{\varepsilon_{n,k}}(u_{n,k})\mathrm{d}v_g
			}{
			\int_{M}e_{\varepsilon_{n,k}}(u_{n,k})\mathrm{d}v_g
		}< \eta,
		\quad \mbox{ for all } p \in \Sigma_{\delta/2}
		\mbox{ and } (n,k) \in \mathbb{N}^2.
	\end{equation*}
	As a consequence,
	\begin{equation}\label{eq:inequalities_energies_nk}
		\frac{\int_{B_g(p,\mu\phi_n^{\frac{1}{N-1}})}
			e_{\varepsilon_{n,k}}(u_{n,k})\mathrm{d}v_g
		}{c(\phi_n)}
		< \eta ,
		\quad \mbox{ for all } p \in \Sigma_{\delta/2}
		\mbox{ and } (n,k) \in \mathbb{N}^2.
	\end{equation}
	Using the $\Gamma$--convergence results 
	given in Proposition~\ref{prop:Gamma-equicoer}
	and Proposition~\ref{prop:Gamma-liminf},
	we deduce by passing to the $\liminf$ as $k \to \infty$ in \eqref{eq:inequalities_energies_nk} that for all $n \in \mathbb{N}$ there exists $\partial T_n$ a $(N-2)$--dimensional integral boundary such that
	\begin{equation}\label{eq:limit_quotient_1}
		\frac{\norm{\partial T_n}(B_g(p,\mu\phi_n^{\frac{1}{N-1}}))}{c(\phi_n)}<\eta, 
		\quad \mbox{ for all } p \in \Sigma_{\delta/2},
	\end{equation}
	and
	\begin{equation}\label{eq:estimate_perimeters_Tn}
		\norm{\partial T_n}(M) \leq c(\phi_n).
	\end{equation}
	Moreover, $T_n(\star X)=\phi_n$ for all $n \in \mathbb{N}$.
	By Remark~\ref{rem:firstorder-cphi},
	we have
	\begin{equation}\label{eq:limit_quotient_2}
		\lim_{n \to \infty} \frac{c(\phi_n)}{J_M(\phi_n,X)}=1,
	\end{equation}
	and $\norm{\partial T_n}(M) \geq J_M(\phi_n,X)$,
	so that we can combine \eqref{eq:estimate_perimeters_Tn} and \eqref{eq:limit_quotient_2} to get
	\begin{equation}\label{eq:limit_quotient_3}
		\lim_{n \to \infty} \frac{\norm{\partial T_n}(M)}{J_M(\phi_n,X)}=1.
	\end{equation}

	By \eqref{eq:limit_quotient_3}, we can apply
	Theorem~\ref{cor:almost-iso-concentration} which yields that
	there exists a sequence of points $(p_n)_{n \in \mathbb{N}}\subset M$ 
	such that $\operatorname{dist}_g(p_n,\Sigma) \to 0$ and
	\begin{equation}
		\lim_{n \to \infty} \frac{\norm{\partial T_n}(B_g(p_n,\mu\phi_n^{\frac{1}{N-1}}))}{\norm{\partial T_n}(M)}=1.
	\end{equation}
	Then,  $p_n \in \Sigma_{\delta/2}$
	for all $n$ sufficiently large, 
	and by using~\eqref{eq:limit_quotient_3} we obtain
	\begin{equation}
		\lim_{n \to \infty} \frac{\norm{\partial T_n}(B_g(p_n,\mu\phi_n^{\frac{1}{N-1}}))}{c(\phi_n)}=1,
	\end{equation}
	which contradicts \eqref{eq:limit_quotient_1} and completes the proof.
\end{proof}

Thanks to Lemma~\ref{lem:jacobian-concentration}, 
we can define the barycenter of a function
$u \in E_{\varepsilon,\phi}^{c(\phi)}$ 
as the \emph{intrinsic barycenter} of its energy density 
$e_\varepsilon(u) \in L^1(M,\mathbb{R})$,
in the sense introduced by Petean in~\cite{MR3912791}, 
building on earlier ideas of Karcher and Grove 
(cf.~\cite{MR442975}).
Loosely speaking, since Lemma~\ref{lem:jacobian-concentration} 
ensures that the energy densities of functions in 
$E_{\varepsilon,\phi}^{c(\phi)}$ are concentrated in small balls,
the barycenter of a function can be effectively computed
by considering only the contribution inside such a ball, while neglecting contributions from the outside.
For the reader’s convenience, we recall the intrinsic barycenter construction;
note that $r_0 > 0$ was chosen so that
every geodesic ball of radius $r < r_0$ is strongly convex.
\begin{lemma}
	\label{lem:intrinsicbarycenterPetean}
	\emph{(Intrinsic barycenter map, cf.~\cite[Theorem 5.2]{MR3912791})}
	Let $r \in (0,r_0/2)$ and $\eta \in (1/2,1)$, and define
	\[
		L^1_{r,\eta}(M) \coloneqq
		\bigg\{
			h \in L^1(M)\setminus \{0\}:
			\sup_{p \in M}
			\frac{\int_{B_g(p,r)} |h|\,\mathrm{d}v_g}{\norm{h}_{L_1(M)}}
			> \eta
		\bigg\}.
	\]
	Then, there exists a continuous function 
	$\mathcal{C}\colon L^1_{r,\eta}(M) \to M$
	such that if $p \in M$ verifies 
	\[
		\frac{\int_{B_g(p,r)} |h|\,\mathrm{d}v_g}{\norm{h}_{L^1(M)}}
		> \eta,
	\]
	then $\mathcal{C}(h) \in B_g(p,2 r)$.
\end{lemma}

\begin{proof}[Proof of Proposition~\ref{prop:barycenter}]
	We fix $\eta \in (1/2,1)$
	and choose $\phi \in (0,\phi_3)$ and
	$\varepsilon \in (0,\varepsilon_3(\phi))$ 
	sufficiently small according
	to Lemma~\ref{lem:jacobian-concentration}.
	Then, for every $u \in E_{\varepsilon,\phi}^{c(\phi)}$
	we have that $e_\varepsilon(u)$
	belongs to $L^1_{r_{\phi},\eta}(M)$
	with $r_{\phi} = \mu \phi^{\frac{1}{N-1}}$.
	Now, recalling the definitions of $\delta $
	in~\ref{ass:Sigmadelta-retraction}
	and of $r_0$ in~\eqref{eq:def-r0},
	let us choose $\phi_2 \in (0,\phi_3)$
	such that
	\[
		r_{\phi_2} = \mu \phi_2^{\frac{1}{N-1}}
		< \frac{1}{2}\min\big\{\delta/2,r_0,\phi_3\big\}.
	\]
	Hence, for any $\phi \in (0,\phi_2)$,
	Lemma~\ref{lem:intrinsicbarycenterPetean} yields a continuous map
	$\mathcal{C}\colon L^1_{r_\phi,\eta}(M) \to M$
	such that $\mathcal{C}(h) \in B_g(p,2r_\phi)$ whenever
	$\int_{B_g(p,r_\phi)} |h|\,\mathrm{d}v_g > \eta \norm{h}_{L^1(M)}$
	for some $p \in M$.
	As a consequence, we define the map
	\[
		\beta^*\colon E_{\varepsilon,\phi}^{c(\phi)} \to M
		\quad\text{by}\quad
		\beta^*(u) = \mathcal{C}(e_\varepsilon(u)),
	\]
	which is continuous by composition.
	By~Lemma~\ref{lem:jacobian-concentration},
	for any $u \in E_{\varepsilon,\phi}^{c(\phi)}$
	there exists $p \in \Sigma_{\delta/2}$ 
	such that $\mathcal{C}(e_{\varepsilon}(u)) \in B_g(p,2r_\phi)$.
	Therefore, by the above choice of $\phi_2$
	we have 
	\[
		\operatorname{dist}_g(\mathcal{C}(u) , \Sigma)
		\le 
		\operatorname{dist}_g(\mathcal{C}(u) , p) +
		\operatorname{dist}_g(p, \Sigma)
		\le 2r_{\phi} + \delta/2
		\le 2\mu \phi_2^{\frac{1}{N-1}}  + \delta/2
		\le \delta,
	\]
	which means that $\beta^*(E_{\varepsilon,\phi}^{c(\phi)})
	\subset \Sigma_{\delta}$.

	Finally, recalling~\ref{ass:Sigmadelta-retraction},
	we define the barycenter map 
	$\beta:E_{\varepsilon,\phi}^{c(\phi)} \to \Sigma$
	by composing $\beta^*$
	with the retraction $h_\Sigma:[0,1]\times \Sigma_\delta \to \Sigma_\delta$,
	namely,
	\begin{equation}
		\label{eq:def-barycenter}
		\beta(u) \coloneqq h_\Sigma(1,\beta^*(u)) \in \Sigma,
	\end{equation}
	which is again continuous by composition.
\end{proof}

\section{Proof of Theorem~\ref{theorem:main}}
\label{sec:finalproof}

By Propositions~\ref{prop:photography-sublevel} and \ref{prop:barycenter}, 
for $\phi,\varepsilon$ sufficiently small the composition
\[
	\beta \circ f_{\varepsilon,\phi} \colon \Sigma \to \Sigma
\]
is well-defined and continuous.
To invoke the photography method (Theorem~\ref{theorem:abstract-photography}), 
it therefore suffices to prove that 
$\beta \circ f_{\varepsilon,\phi}$ is homotopic to the identity on $\Sigma$.
Once this homotopy is established, 
all the hypotheses of Theorem~\ref{theorem:abstract-photography} 
are satisfied for the Ginzburg--Landau functional $E_{\varepsilon,\phi}$,
and Theorem~\ref{theorem:main} follows.
We now turn to the formal proof.

\begin{proof}
	Recalling the thresholds $\phi_1,\varepsilon_1(\phi)$
	from Proposition~\ref{prop:photography-sublevel}
	and $\phi_2,\varepsilon_2(\phi)$ from Proposition~\ref{prop:barycenter}
	(with $\phi_2<\phi_1$ and $\varepsilon_2(\phi) < \varepsilon_1(\phi)$), we have that,
	for any $\phi<\phi_2$ and $\varepsilon< \varepsilon_2(\phi)$, 
	the image of the photography map $f_{\varepsilon,\phi}$ lies in the sublevel 
	$E_{\varepsilon,\phi}^{c(\phi)}$ on which the barycenter map $\beta$ is well defined.
	In what follows, we always consider $\phi<\phi_2$
	and $\varepsilon<\varepsilon_2(\phi)$,
	without further comment.    

	We construct the required homotopy 
	$F\colon [0,1]\times \Sigma \to \Sigma$
	between the identity and $\beta \circ f_{\varepsilon,\phi}$
	as follows:
	\begin{equation}
		\label{eq:FinalHomotopy}
		F(t,p) \coloneqq h_\Sigma\bigg(1,\exp_p\!\Big(t \exp_p^{-1}
				\big(\beta(f_{\varepsilon,\phi}(p))\big)
		\Big)\bigg).
	\end{equation}
	Since $h_\Sigma(1,p) = p$ for every $p \in \Sigma$,
	we have that 
	$F(0,\cdot) = \operatorname{Id}$ and
	$F(1,p) = \beta(f_{\varepsilon,\phi}(p))$ for every $p \in \Sigma$,
	as long as $F$ is well defined.
	Then, for the above considerations, Theorem~\ref{theorem:main}
	is proved by applying Theorem~\ref{theorem:abstract-photography}.

	It remains to check that $F$ is well defined for $\phi,\varepsilon$ sufficiently small.
	To this end, it suffices to show both that 
	\begin{equation}
		\label{eq:finalhomotopy-intermediate1}
		\operatorname{dist}_g\big(p,\beta(f_{\varepsilon,\phi}(p))\big)
		\le r_0 < \operatorname{inj}(M),
		\qquad \forall p \in \Sigma,        
	\end{equation}
	which ensures that the inverse of $\exp_p$ 
	that appears in~\eqref{eq:FinalHomotopy} can be applied,
	and that  
	\begin{equation}
		\label{eq:finalhomotopy-intermediate2}
		\exp_p\!\Big(t \exp_p^{-1}
			\big(\beta(f_{\varepsilon,\phi}(p))\big)
		\Big)
		\in \Sigma_\delta,
		\qquad \forall (t,p) \in [0,1]\times \Sigma.
	\end{equation}
	In particular, this last property implies that 
	the second argument of $h_\Sigma$ in~\eqref{eq:FinalHomotopy} lies in $\Sigma_\delta$, so $F(t,p)\in\Sigma$ for all $(t,p)\in[0,1]\times\Sigma$.
	Moreover, since for $v\in T_pM$ with $\|v\|_g<r_0$ one has
	\[
		\operatorname{dist}_g\big(p,\exp_p(t v)\big)=t\,\|v\|_g,
		\qquad \forall t \in [0,1],
	\]
	both \eqref{eq:finalhomotopy-intermediate1} and \eqref{eq:finalhomotopy-intermediate2}
	follow if we prove that 
	\begin{equation}
		\label{eq:finalhomotopy-intermediate3}
		\operatorname{dist}_g\big(p,\beta(f_{\varepsilon,\phi}(p))\big)
		\le \rho^* \coloneqq \tfrac12\min\{r_0,\delta\},
		\qquad \forall p \in \Sigma.        
	\end{equation}

	By the explicit construction of the photography map
	(see Lemma~\ref{lem:photoConstruction-hat})
	we know that
	\begin{equation}
		\label{eq:finalHomotopy-proof1}
		\operatorname{supp} e_\varepsilon\big(f_{\varepsilon,\phi}(p)\big)
		\subset D_X\big(p,2r(p,\phi)\big),
		\qquad \forall p \in \Sigma.
	\end{equation}
	By Lemma~\ref{lem:rightOrthoDisk},
	$r(p,\phi)$ is a continuous function 
	and $r(p,0) = 0$ for any $p \in \Sigma$.
	By the compactness of $\Sigma$, 
	we then obtain that 
	for any $r \in (0,r_0)$
	there exists $\phi_r \in (0,\phi_2)$
	such that 
	\[
		r(p,\phi) \le \frac{r}{4} < \frac{r_0}{4},
		\qquad \forall (p,\phi) \in \Sigma \times (0,\phi_r).
	\]
	Therefore, by Lemma~\ref{lem:intrinsicbarycenterPetean}
	and~\eqref{eq:finalHomotopy-proof1},
	for any $\eta \in (1/2,1)$
	and $\phi \in (0,\phi_r)$, we have
	\[
		e_\varepsilon\big(f_{\varepsilon,\phi}(p)\big)
		\in L^1_{r/2,\eta},
		\qquad \forall p \in \Sigma,
	\]
	and 
	\begin{equation}
		\label{eq:finalHomotopy-proof2}
		\beta^*\big(f_{\varepsilon,\phi}(p)\big) = \mathcal{C}\big(e_\varepsilon(f_{\varepsilon,\phi}(p))\big)
		\in B_g(p,r) \subset B_g(p,r_0).
	\end{equation}

	Since $h_\Sigma(1,\cdot):\Sigma_\delta \to \Sigma$
	is a continuous map on a compact set and
	$h_\Sigma(1,p) = p$ for every $p \in \Sigma$,
	for any $\rho > 0$ there exists $r_\rho \in (0,\delta)$
	such that for any $p \in \Sigma$ we have
	\begin{equation}
		\label{eq:finalHomotopy-proof3}
		\operatorname{dist}_g\big(p,h_\Sigma(1,q)\big)
		< \rho, \qquad \forall q \in B_g(p,r_\rho).
	\end{equation}
	Set
	\[
		r^* = r_{\rho^*},
		\qquad 
		\phi^* = \phi_{r^*},
		\qquad \text{and }
		\varepsilon^*(\phi) = \varepsilon_2(\phi).
	\]
	Recalling that $\beta(u)=h_\Sigma(1,\beta^*(u))$,
	combining~\eqref{eq:finalHomotopy-proof2}
	with~\eqref{eq:finalHomotopy-proof3} leads to 
	\[
		\operatorname{dist}_g\big(p,\beta(f_{\varepsilon,\phi}(p))\big)
		=\operatorname{dist}_g\Big(p,h_\Sigma\big(1,\beta^*(f_{\varepsilon,\phi}(p))\big)\Big)
		<\rho^*,
	\]
	for every $\phi\in(0,\phi^*)$
	and $\varepsilon\in(0,\varepsilon^*(\phi))$,
	thus proving \eqref{eq:finalhomotopy-intermediate3}.
	Therefore,
	$\beta\circ f_{\varepsilon,\phi}$
	is homotopic to $ \mathrm{id}_\Sigma$, as claimed.
\end{proof}

\appendix
\section{On the behavior of sequences of almost minimizers of the isoperimetric problems}
\label{app:generalized-compactness}

In this appendix we show
that, for small flux parameters, almost minimizers
of the isoperimetric problem in higher codimension 
on a compact Riemannian manifold
concentrate almost all their mass within a single small ball. Identity \eqref{eq:flux_coexact}, which shows that the flux depends only on the boundary of the integral current, will be used consistently. Given $a>0$, we will use the symbol $o(a)$ to denote a quantity depending also on $M$ and $X$ and such that
\begin{equation*}
	\lim_{a \to 0^+}\frac{\lvert o(a) \rvert}{\lvert a \rvert} = 0.
\end{equation*}

\begin{theoremletter}[Concentration for almost isoperimetric regions]
	\label{cor:almost-iso-concentration}
	There exists a quantity $\mu > 0$,
	depending only on $(M,g)$ and $X$,
	such that the following property holds:
	For every sequence $(T_n)_{n \in \mathbb{N}}$ of
	$(N-1)$--dimensional integral currents such that 
	$\phi_n = T_n(\star X)\to 0$ and
	\begin{equation}
		\label{eq:almost-minimizer-condition}
		\lim_{n\to\infty}
		\frac{\norm{\partial T_n}(M)}{J_M(\phi_n, X)} = 1,
	\end{equation}
	there exists a sequence of points in
	$(p_n)_{n \in \mathbb{N}}\subset M$ such that
	$\lim_{n \to \infty}\mathrm{dist}(p_n,\Sigma)=0$
	and
	\[
		\lim_{n \to \infty}
		\frac{
		\norm{\partial T_n}(	B_g(p_n,\mu\,\phi_n^{\frac{1}{N-1}}))}
		{\norm{\partial T_n}(M)} = 1.
	\]
\end{theoremletter}
The proof of Theorem \ref{cor:almost-iso-concentration} is divided into several intermediate results.
\begin{lemma}[A lower bound for the perimeter]\label{lemma:small_per}
	Let $p \in M$, $C_{\mathrm{flux}} \geq 1$ and $\mu_{\mathrm{max}} \geq 1$. For all $\phi >0$ such that $2\mu_{\max} \phi^{\frac{1}{N-1}} \leq \operatorname{inj}(M)$, $\mu \in (0,\mu_{\max}]$ and $S \in \mathcal{IB}_{N-2}(M)$ such that $\mathrm{spt}(S) \cap  B_g(p,\mu\phi^{\frac{1}{N-1}})$ is connected up to a $\mathcal{H}^{N-2}$-negligible set and
	\begin{equation}\label{eq:per_small_flux}
		S|_{B_g(p,\mu \phi^{\frac{1}{N-1}})}(Y)= \phi_* \in [-C_{\mathrm{flux}}\phi,C_{\mathrm{flux}}\phi]\setminus \{0\}
	\end{equation}
	one has
	\begin{equation}\label{eq:per_small}
		\norm{X(p)}^{\frac{N-2}{N-1}}\norm{S}(B_g(p,\mu\phi^{\frac{1}{N-1}})) \geq \gamma_{N-1}\lvert\phi_*\rvert^{\frac{N-2}{N-1}}+o(\phi^{\frac{N-2}{N-1}}).
	\end{equation}
\end{lemma}
\begin{proof}
	Since the flux constraint is linear and the perimeter is invariant by orientation, one can assume for simplicity that $\phi_*$ is positive. Let $\tilde{S}$ be the restriction of $S$ to $B_g(p,\mu  \phi^{\frac{1}{N-1}})$.
	Since $\mathrm{spt}(S) \cap  B_g(p,\mu\phi^{\frac{1}{N-1}})$
	is connected up to a $\mathcal{H}^{N-2}$-negligible set, one has that $\tilde{S}$ is a $(N-2)$-dimensional integral boundary with the same mass as $S$ on $B_g(p,\mu\phi^{\frac{1}{N-1}})$. Let 
	\[
		\varphi_p: B_g(p,\operatorname{inj}(M)) \to 
		T_pM \cong \mathbb{R}^N
	\]
	be a diffeomorphism (of the same regularity as $M$)
	onto its image, with $\varphi_p(p) = 0$. Consider now $S_{p}:= \varphi_{p\#}\tilde{S}$, which (after extending by zero) is a $(N-2)$-dimensional integral boundary on $\mathbb{R}^N$.
	Let $X_{p}$ be the constant vector field in $\mathbb{R}^N$ induced by $X(p)$.
	Thus, $\star X_p$ is a closed form and since $\mathbb{R}^N$ is contractible, we find $Y_p$ a $(N-2)$-form on $\mathbb{R}^N$ such that $\mathrm{d}Y_p=\star X_p$. Then,
	\begin{equation}\label{eq:quotient_difference_flux}
		S_p(Y_p)-\phi_*=S_p( Y_{p})-\tilde{S}(Y)=\tilde{S}(\varphi_p^\#( Y_{p})-Y).
	\end{equation}
	Notice that by \eqref{eq:per_small_flux} one can assume that
	\begin{equation}\label{eq:mass_bd_balls}
		\lVert S \rVert(B_g(p,\mu\phi^{\frac{1}{N-1}})) \leq C_{\mathrm{M-per}}\phi^{\frac{N-2}{N-1}}
	\end{equation}
	for some positive quantity $C_{\mathrm{M-per}}$, as for otherwise \eqref{eq:per_small} holds and there is nothing else to proof. By plugging \eqref{eq:mass_bd_balls} into \eqref{eq:quotient_difference_flux} we obtain
	\begin{equation}\label{eq:diff_flux_1}
		\lvert S_p( Y_p)-\phi_* \rvert \leq C_{\mathrm{M-per}}\phi^{\frac{N-2}{N-1}}\norm{\varphi_p^\#( Y_p)-Y}_{L^\infty(B_g(p,\mu\phi^{\frac{1}{N-1}}))}.
	\end{equation}
	Since $Y$ is smooth and $M$ is compact, we have
	\begin{equation*}
		\norm{\varphi_p^\#(Y_p)-Y}_{L^\infty(B_g(p,\mu\phi^{\frac{1}{N-1}}))} = o(\phi^{\frac{1}{N-1}}).
	\end{equation*}
	Therefore, \eqref{eq:diff_flux_1} becomes
	\begin{equation}\label{eq:diff_flux_2}
		\lvert S_p(Y_p)-\phi_* \rvert=o(\phi).
	\end{equation}
	The isoperimetric inequality for integral currents (see Almgren~\cite{Almgren1986}) implies that there exists $T_p$ a $(N-1)$-dimensional integral current on $\mathbb{R}^N$ with $\partial T_p=S_p$ and
	\begin{equation}\label{eq:iso_RN}
		\norm{S_p}(\mathbb{R}^N) \geq \gamma_{N-1}\left(\norm{T_p}(\mathbb{R}^N)\right)^{\frac{N-2}{N-1}}.
	\end{equation}
	Notice then that \eqref{eq:diff_flux_2} becomes
	\begin{equation}\label{eq:diff_flux_def}
		\lvert T_p(\star X_p)-\phi_* \rvert = o(\phi).
	\end{equation}
	From \eqref{eq:iso_RN} and \eqref{eq:diff_flux_def} we obtain
	\begin{equation}\label{eq:S_p_per_phi}
		\norm{X(p)}^{\frac{N-2}{N-1}}\norm{S_p}(\mathbb{R}^N)\geq \gamma_{N-1} \lvert \phi_* \rvert^{\frac{N-2}{N-1}}+o(\phi^{\frac{N-2}{N-1}}).
	\end{equation}
	We have that (cf.~\cite[Proposition 3, Page 133]{GMS-VolI})
	\begin{equation}\label{eq:mass_pullback}
		\norm{S_p}(\mathbb{R}^{N}) \leq \norm{\nabla f_p}_{L^\infty(B_g(p,\mu\phi^{\frac{1}{N-1}}))}^{N-2}\norm{S}(B_g(p,\mu \phi^{\frac{1}{N-1}})).
	\end{equation}
	Since \eqref{eq:mass_bd_balls} holds and $f$ is smooth, inequality \eqref{eq:mass_pullback} implies
	\begin{equation*}
		\norm{S_p}(\mathbb{R}^{N}) \leq \norm{S}(B_g(p,\mu \phi^{\frac{1}{N-1}}))+o(\phi^{\frac{N-2}{N-1}}),
	\end{equation*}
	which plugged into \eqref{eq:S_p_per_phi} gives \eqref{eq:per_small}.
\end{proof}
The results below will concern integral boundaries $S \in \mathcal{IB}_{N-2}(M)$ satisfying the perimeter bound
\begin{equation}\label{eq:dec_bound_per}
	\norm{S}(M) \leq C_{\mathrm{per}} \phi^{\frac{N-1}{N-2}},
\end{equation}
for some quantity $C_{\mathrm{per}}>0$ and $\phi>0$.
\begin{lemma}[Bounding connected components of the support]\label{lemma:dec}
	Let $C_{\mathrm{per}}>0$. There exists $\mu_{\mathrm{dec}} \geq 1$ such that for all $\phi>0$ satisfying $2\mu_{\mathrm{dec}}\phi^{\frac{1}{N-1}} \leq \operatorname{inj}(M)$ and $S \in \mathcal{IB}_{N-2}(M)$ satisfying \eqref{eq:dec_bound_per} one has that each connected component of $\mathrm{spt}(S)$ is contained by some ball of radius $\mu_{\mathrm{dec}}\phi^{\frac{1}{N-1}}$ up to a $\mathcal{H}^{N-2}$-negligible set.
\end{lemma}
\begin{proof}
	Since $M$ is compact, we can choose $\mu_{\mathrm{dec}} \geq 1$ large enough so that for all $\phi$ positive such that $2\mu_{\mathrm{dec}}\phi^{\frac{1}{N-1}}\leq \operatorname{inj}(M)$ and $A$ a $(N-2)$-dimensional rectifiable connected subset of $M$ such that $\mathcal{H}^{N-2}(A) \leq C_{\mathrm{per}}\phi^{\frac{N-2}{N-1}}$ one has $A \subset B_g(p_A,\mu_{\mathrm{dec}}\phi^{\frac{1}{N-1}})$ up to a $\mathcal{H}^{N-2}$-negligible set for some $p_A \in M$. Let now $\phi$ and $S$ be as in the statement. Since $S$ is rectifiable, inequality \eqref{eq:dec_bound_per} yields
	\begin{equation*}
		\mathcal{H}^{N-2}(\mathrm{spt}(S)) \leq C_{\mathrm{per}} \phi^{\frac{N-1}{N-2}}.
	\end{equation*}
	As a consequence, each connected component of $\mathrm{spt}(S)$ is contained by some ball of radius $\mu_{\mathrm{dec}}\phi^{\frac{1}{N-1}}$ up to a $\mathcal{H}^{N-2}$-negligible set, as we wanted to prove.
\end{proof}
\begin{lemma}[Bounding the flux by the perimeter]\label{lemma:bound_flux_per}
	Let $C_{\mathrm{per}}>0$. Then, for all $\phi>0$ and $S \in \mathcal{IB}_{N-2}(M)$ such that \eqref{eq:dec_bound_per} holds one has $\lvert S(Y) \rvert \leq C_{\mathrm{flux}}\phi$ for some quantity $C_{\mathrm{flux}}>0$ depending only on $M$ and $C_{\mathrm{per}}$.
\end{lemma}
\begin{proof}
	By using the isoperimetric inequality on $M$ (see \cite[4.4]{Federer1996}), a $(N-1)$-dimensional integral current $T$ such that $\partial T =S$ and
	\begin{equation}
		\gamma_M \left(\norm{T}(M)\right)^{\frac{N-2}{N-1}} \leq \norm{S}(M)
	\end{equation}
	for some quantity $\gamma_{M}>0$ depending only on $M$. The result then follows from \eqref{eq:dec_bound_per} by recalling that $T(\star X)=S(Y)$.
\end{proof}
\begin{lemma}[Decomposition and global lower bound on the perimeter]\label{lemma:dec_global}
	Let $C_{\mathrm{per}}>0$ be fixed and $\mu_{\mathrm{dec}} \geq 1$ be the corresponding quantity given by Lemma \ref{lemma:dec}. Then, for all $\phi>0$ such that $2\mu_{\mathrm{dec}}\phi^{\frac{1}{N-1}} \leq \operatorname{inj}(M)$ and $S \in \mathcal{IB}_{N-2}(M)$ satisfying \eqref{eq:dec_bound_per} there exists $(S_n)_{n \in \mathbb{N}}$ in $\mathcal{IB}_{N-2}(M)$ such that $(\mathrm{spt}(S_n))_{n \in \mathbb{N}}$ is a sequence of pairwise disjoint rectifiable sets with $\mathrm{spt}(S_n)$ connected, $S=\sum_{n=0}^{+\infty}S_n$  up to a $\mathcal{H}^{N-2}$-negligible set and
	\begin{equation}\label{eq:per_global}
		\norm{S}(M) \geq \gamma_{N-1}\sum_{n=0}^{+\infty} \lvert S_n(Y) \rvert^{\frac{N-2}{N-1}} + o(\phi^{\frac{N-2}{N-1}}).
	\end{equation}
\end{lemma}
\begin{proof}
	Since $S$ satisfies \eqref{eq:dec_bound_per}, we can apply Lemma \ref{lemma:dec} and deduce that each connected component of $\mathrm{spt}(S)$ is contained in some ball of radius $\mu_{\mathrm{dec}}\phi^{\frac{1}{N-1}}$. Moreover, as $\mathrm{spt}(S)$ is rectifiable, there are most countably many of such components. Therefore, we find $(S_n)_{n \in \mathbb{N}}$ such that $S_n$ is an integral boundary with (possibly empty) connected support for all $n \in \mathbb{N}$ and $S=\sum_{n=0}^{+\infty} S_n$ up to a $\mathcal{H}^{N-2}$-negligible set. Let us now denote $\phi_n:=\left(C_{\mathrm{per}}^{-1}\norm{S_n}(M)\right)^{\frac{N-2}{N-1}}$ for all $n \in \mathbb{N}$. Lemmas \ref{lemma:dec} and \ref{lemma:bound_flux_per} imply that we can apply Lemma \ref{lemma:small_per} to $\phi_n$ with $S_n$ and get
	\begin{equation}\label{eq:per_Sk}
		\norm{S_n}(M) \geq \gamma_{N-1}\lvert S_n(Y) \rvert+o(\phi_n^{\frac{N-1}{N-2}}).
	\end{equation}
	By noticing that
	\begin{equation*}
		\left\lvert \sum_{n=0}^{+\infty}o(\phi_n^{\frac{N-1}{N-2}})\right\rvert=\left\lvert\sum_{n=0}^{+\infty}o(\norm{S_n}(M))\right\rvert \leq \left\lvert o(\norm{S})\right\rvert \leq \lvert o(\phi^{\frac{N-1}{N-2}}) \rvert,
	\end{equation*}
	inequality \eqref{eq:per_global} follows by adding for $n \in \mathbb{N}$ in \eqref{eq:per_Sk}.
\end{proof}
\begin{proposition}[First order expansion of the isoperimetric profile]\label{prop:iso_expansion}
	Let $\mu_{\mathrm{dec}}$ be the quantity given by Lemma \ref{lemma:dec} for $C_{\mathrm{per}}=\gamma_{N-1}+1$. Then, for all $\phi>0$ such that $2\mu_{\mathrm{dec}}\phi^{\frac{1}{N-1}} \leq \operatorname{inj}(M)$ one has
	\begin{equation}\label{eq:iso_expansion}
		J_M(\phi,X)= \gamma_{N-1} \phi^{\frac{N-1}{N-2}}+o(\phi^{\frac{N-1}{N-2}}).
	\end{equation}
\end{proposition}
\begin{proof}
	From Lemma \ref{lem:estimate-partialDphi} we readily obtain
	\begin{equation*}
		J_M(\phi,X) \leq \gamma_{N-1}\phi^{\frac{N-2}{N-1}}+o(\phi^{\frac{N-2}{N-1}}).
	\end{equation*}
	In order to complete the proof, we need to prove that an arbitrary $S \in \mathcal{IB}_{N-2}^{\phi}(M)$ satisfies
	\begin{equation}\label{eq:lb_iso_profile}
		\norm{S}(M) \geq \gamma_{N-1}\phi^{\frac{N-2}{N-1}}+o(\phi^{\frac{N-2}{N-1}}).
	\end{equation}
	If $\norm{S}(M) > (\gamma_{N-1}+1)\phi^{\frac{1}{N-1}}$, then \eqref{eq:lb_iso_profile} readily holds. Otherwise, we can apply Lemma \ref{lemma:dec_global} and \eqref{eq:lb_iso_profile} follows.
\end{proof}
The following properties are elementary and probably well-known:
\begin{lemma}[An elementary subadditivity property]\label{lemma:subadd}
	Let $a,b$ and $c$ be positive numbers such that $c=a+b$. Assume that there exists $\delta>0$ such that
	$a \in [\delta,c-\delta]$. Then, for all $s \in (0,1)$ 
	\begin{equation}\label{eq:abc_s}
		a^s+b^s \geq c^s+s(1-s)c^{s-2}\delta^2.
	\end{equation}
\end{lemma}
\begin{proof}
	Consider the function
	\begin{equation*}
		p_s: x \in [0,+\infty) \mapsto  p_s(x):= x^s \in [0,+\infty).
	\end{equation*}
	Notice that
	\begin{equation*}
		p_s''(x)=s(s-1)x^{s-2} \leq -s(1-s)c^{s-2} \quad \mbox{ for all } x \in (0,c].
	\end{equation*}
	As a consequence, the function
	\begin{equation*}
		p_{s,c}: [0,c] \to [0,+\infty), \quad \rho_{s,c}(x):= p_s(x)+\frac{s(1-s)c^{s-2}}{2}x^2 
	\end{equation*}
	is concave. Therefore, $p_{s,c}$ is also subadditive, which means
	\begin{equation*}
		p_{s,c}(a)+p_{s,c}(b) \geq p_{s,c}(c),
	\end{equation*}
	or, equivalently
	\begin{equation*}
		a^s+b^s+\frac{s(1-s)c^{s-2}}{2}(a^2+b^2) \geq c^s +\frac{s(1-s)c^{s-2}}{2}c^2,
	\end{equation*}
	which means
	\begin{equation}\label{eq:abc_0}
		a^s+b^s \geq c^s + s(1-s)c^{s-2}ab.
	\end{equation}
	Notice that since $a \in [\delta,c-\delta]$ and $c=a+b$, one has that $b \in [\delta,c-\delta]$. As a consequence, \eqref{eq:abc_s} follows from \eqref{eq:abc_0}.
\end{proof}
\begin{lemma}[An elementary subadditivity property for sequences of series]\label{lemma:series}
	Let $(a_{n,m})_{(n,m) \in \mathbb{N}^2}$ be such that $a_{n,m}$ is nonnegative for all $(n,m) \in \mathbb{N}^2$. Let $s \in (0,1)$ and assume that there exist positive quantities $c_{-} < c_{+}$ such that 
	\begin{equation}\label{eq:series_asu2}
		c_{n}:= \sum_{m=0}^{+\infty} a_{n,m} \in [c_{-},c_{+}] \quad \mbox{ for all } n \in \mathbb{N}
	\end{equation}
	and
	\begin{equation}\label{eq:series_asu3}
		\limsup_{n \to \infty}\sup_{m \in \mathbb{N}}a_{n,m} \leq 1-\delta.
	\end{equation}
	Then,
	\begin{equation}\label{eq:series_concl}
		c_{s,n} \geq c_{-}^s+s(1-s)c_{+}^{s-2}\delta^2 \quad \mbox{ for all } n \in \mathbb{N}.
	\end{equation}
\end{lemma}
\begin{proof}
	As a consequence of \eqref{eq:series_asu3}, for all $n \in \mathbb{N}$ large enough, there exists $A_n \subset \mathbb{N}$ such that
	\begin{equation}\label{eq:ratio_sum}
		a_n:=\sum_{m \in A_n} a_{n,m} \in [\delta,1-\delta].
	\end{equation}
	Therefore, by using the subadittivity of $x \mapsto x^{s}$ along with \eqref{eq:series_asu2}, we get
	\begin{equation}\label{eq:rho_n_ineq}
		c_{s,n} \geq a_n^{s}+(c_n-a_n)^{s} \quad \mbox{ for all } n \in \mathbb{N} \mbox{ large enough}
	\end{equation}
	By \eqref{eq:ratio_sum}, for all $n \in \mathbb{N}$ we can apply Lemma \ref{lemma:subadd} with $c=c_n$, $a=a_n$, $b=c_n-a_n$. Inequality \eqref{eq:rho_n_ineq} then becomes
	\begin{equation*}
		c_{s,n} \geq c_{n}^{s}+s(1-s)c_n^{s-2}\delta^2 \quad \mbox{ for all } n \in \mathbb{N} \mbox{ large enough}
	\end{equation*}
	and \eqref{eq:series_concl} follows then by \eqref{eq:series_asu2}.
\end{proof}
\begin{proof}[Proof of Theorem \ref{cor:almost-iso-concentration}]
	Let $(\phi_{n})_{n \in \mathbb{N}}$ be a sequence of positive numbers tending to zero. Let $(S_n)_{n \in \mathbb{N}}$ be a sequence in $\mathcal{IB}_{N-2}^{\phi}(M)$ such that $S_n(Y)=\phi_n$ and
	\begin{equation}\label{eq:almost_optimal_limit}
		\lim_{n \to \infty}\frac{\norm{S_n}(M)}{J_M(\phi,X)}=1.
	\end{equation}
	According to \eqref{eq:almost_optimal_limit} and Proposition \ref{prop:iso_expansion}, we have that $\norm{S_n}(M) \leq (\gamma_{N-1}+1)\phi_{n}^{\frac{N-2}{N-1}}$ for all $n$ large enough. Let $\mu_{\mathrm{dec}}$ be the quantity given by Lemma \ref{lemma:dec} for $C_{\mathrm{per}}=\gamma_{N-1}+1$. For all $n$ large enough one has that $2\mu_{\mathrm{dec}}\phi_n^{\frac{1}{N-1}} \leq \operatorname{inj}(M)$ and Lemma \ref{lemma:dec_global} implies that for all such $n$
	\begin{equation*}
		\norm{S_n}(M) \geq \gamma_{N-1}\sum_{m=0}^{+\infty}\lvert S_{n,m}(Y) \rvert^{\frac{N-2}{N-1}}+o(\phi_{n}^{\frac{N-2}{N-1}}),
	\end{equation*}
	where $S_{n,m}$ has connected support, $\mathrm{spt}(S_{n,m_1}) \cap \mathrm{spt}(S_{n,m_2})=\emptyset$ up to a $\mathcal{H}^{N-2}$-negligible set and $S_n=\sum_{m=0}^{+\infty}S_{n,m}$. From \eqref{eq:almost_optimal_limit} we then obtain
	\begin{equation}\label{eq:ratio_ineq}
		1 \geq \lim_{n \to \infty}\rho_n.
	\end{equation}
	where
	\begin{equation*}
		\rho_n:= \frac{\sum_{m=0}^{+\infty}\lvert S_{n,m}(Y) \rvert^{\frac{N-2}{N-1}}}{\phi_n^{\frac{N-2}{N-1}}} \quad \mbox{ for all } n \in \mathbb{N}
	\end{equation*}
	Notice also that for all $n$ large enough, the subadittivity of the function  $x^{\frac{N-2}{N-1}} \in [0,+\infty)$ implies
	\begin{equation}\label{eq:sum_flux_ineq}
		\sum_{m=0}^{+\infty}\lvert S_{n,m}(Y) \rvert^{\frac{N-2}{N-1}} \geq \left(\sum_{m=0}^{+\infty}\lvert S_{n,m}(Y) \rvert\right)^{\frac{N-2}{N-1}} \geq \phi_n^{\frac{N-2}{N-1}}.
	\end{equation}
	By combining \eqref{eq:ratio_ineq} and \eqref{eq:sum_flux_ineq}, we obtain
	\begin{equation}\label{eq:ratio_lim}
		\lim_{n \to \infty}\rho_n=1.
	\end{equation}
	One also checks that
	\begin{equation}\label{eq:ratio_claim}
		\mbox{There exists a sequence } (m_n)_{n \in \mathbb{N}} \mbox{ in } \mathbb{N} \mbox{ such that } \lim_{n \to \infty} \frac{\lvert S_{n,m_n}(Y) \rvert}{\phi_n} =1.
	\end{equation}
	Indeed, assume by contradiction that \eqref{eq:ratio_claim} does not hold. Then, one can apply Lemma \ref{lemma:series} with $(a_{n,m})_{(n,m) \in \mathbb{N}^2}=\left(\frac{\lvert S_{n,m}(Y) \rvert}{\phi_n}\right)_{(n,m) \in \mathbb{N}^2}$ and $s=\frac{N-2}{N-1}$. More precisely, one checks that $c_-=1$ and the existence of $c_+$ follows by Lemmas \ref{lemma:bound_flux_per} and \ref{lemma:dec_global}. In the present setting, \eqref{eq:series_concl} reads (for some $\delta>0$)
	\begin{equation}
		\rho_n \geq 1+\frac{N-2}{(N-1)^2}c_+^{\frac{-N}{N-1}}\delta^2,
	\end{equation}
	which contradicts \eqref{eq:ratio_lim}. Up to rearranging, it can be assumed than $m_n=0$ for all $n \in \mathbb{N}$. From \eqref{eq:ratio_lim} and \eqref{eq:ratio_lim}, we get
	\begin{equation*}
		\lim_{n \to \infty}\frac{\lvert S_{n,m}(Y) \rvert}{\phi_n}=0 \quad \mbox{ for all } m \in \mathbb{N}^*.
	\end{equation*}
	Hence, since $\sum_{m=0}^{+\infty}S_{n,m}(Y)=\phi_n^{\frac{N-2}{N-1}}$ for all $n \in \mathbb{N}$, we in fact have that $S_{n,0}(Y)>0$ for $n$ large enough and, moreover
	\begin{equation}\label{eq:limit_quotient0}
		\lim_{n \to \infty}\frac{S_{n,0}(Y)^{\frac{N-2}{N-1}}}{\phi_n^{\frac{N-2}{N-1}}}=1.
	\end{equation}
	Recall that Lemma \ref{lemma:dec} yields the existence of $p_n \in M$ such that $\mathrm{spt}(S_{n,0}) \subset B_g(p_n,\mu_{\mathrm{dec}}\phi_n^{\frac{1}{N-1}})$. Thus, by applying again Lemma \ref{lemma:small_per}, we obtain for all $n \in \mathbb{N}$
	\begin{equation*}
		\norm{X(p_n)}^{\frac{N-2}{N-1}}\frac{\norm{S}( B_g(p_n,\mu_{\mathrm{dec}}\phi_n^{\frac{1}{N-1}}))}{\phi_n^{\frac{N-2}{N-1}}} \geq \frac{S_{n,0}(Y)^{\frac{N-2}{N-1}}}{\phi_n^{\frac{N-2}{N-1}}}+o_n(1).
	\end{equation*}
	By combining \eqref{eq:almost_optimal_limit} and \eqref{eq:limit_quotient0}, we get
	\begin{equation}\label{eq:lim_Xptilde}
		\lim_{n \to \infty}\norm{X(p_n)}^{\frac{N-2}{N-1}}=1
	\end{equation}
	and
	\begin{equation*}
		\lim_{n \to \infty}\frac{\norm{S}( B_g(p_n,\mu_{\mathrm{dec}}\phi_n^{\frac{1}{N-1}}))}{\phi_n^{\frac{N-2}{N-1}}}=1
	\end{equation*}
	The limit \eqref{eq:lim_Xptilde} implies that $\lim_{n \to \infty}\mathrm{dist}(p_n,\Sigma)=0 $, which completes the proof with $\mu:=\mu_{\mathrm{dec}}$.
\end{proof}

\bibliographystyle{abbrv}
\bibliography{ourbib}

\end{document}